\documentclass{article}

\usepackage{graphicx}

\usepackage{amsmath,amssymb,enumerate,mathrsfs}
\usepackage{theorem}

\newtheorem{theorem}{Theorem}[section]
\newtheorem{corollary}[theorem]{Corollary}
\newtheorem{lemma}[theorem]{Lemma}
\newtheorem{proposition}[theorem]{Proposition}
\theoremstyle{definition}
\newtheorem{definition}[theorem]{Definition}
\newtheorem{example}[theorem]{Example}

\newtheorem{remark}[theorem]{Remark}
\newtheorem{alg}[theorem]{Algorithm}
\newtheorem{fact}[theorem]{Fact}

\usepackage{enumitem}

\numberwithin{equation}{section}

\newcommand{\dom}{\ensuremath{\operatorname{dom}}}
\def\Sdual{{S_{\rm dual}}}
\def\RR{{\mathbb{R}}}

\def\NN{{\mathbb{N}}}
\def\kkk{{k\in\NN}}
\def\mcN{{\mathcal{N}}}
\def\mcF{{\mathcal{F}}}
\def\mcP{{\mathcal{P}}}
\def\disp{\displaystyle}
\newcommand{\scal}[2]{\left\langle{#1},{#2}  \right\rangle}

\def\ox{\overline{x}}

\def\ov{\overline{v}}
\def\tx{\tilde{x}}
\def\dd{\delta}
\def\emp{\varnothing}
\def\ra{\rangle}
\def\la{\langle}
\newcommand{\bdry}{\ensuremath{\operatorname{bd}}}
\def\proof{\noindent{\it Proof}. \ignorespaces}
\def\endproof{\ensuremath{\hfill \quad \blacksquare}}
\newcommand{\Id}{\ensuremath{\operatorname{Id}}}

\newenvironment{retraitsimple}{\begin{list}{--~}{
 \topsep=0.3ex \itemsep=0.3ex \labelsep=0em \parsep=0em
 \listparindent=1em \itemindent=0em
 \settowidth{\labelwidth}{--~} \leftmargin=\labelwidth
}}{\end{list}}

\theoremstyle{plain}{\theorembodyfont{\rmfamily}
\newtheorem{customC}{Conceptual Algorithm}
\newenvironment{Calg}[1]
  {\renewcommand\thecustomC{#1}\customC}
  {\endcustomC}

\theoremstyle{plain}{\theorembodyfont{\rmfamily}
\newtheorem{customL}{Linesearch}
\newenvironment{linesr}[1]
  {\renewcommand\thecustomL{#1}\customL}
  {\endcustomL}

\begin{document}
\makeatletter

\begin{center}
\large{\bf Conditional Extragradient Algorithms for Solving Variational Inequalities}
\end{center}\vspace{5mm}
\begin{center}

\textsc{J.Y. Bello-Cruz\footnote{Department of Mathematical Sciences, Northern Illinois University, 
DeKalb, IL \ 60115, USA.
E-mail: {yunierbello@niu.edu}}, R. D\'iaz Mill\'an \footnote{
Federal Institute Goi\'as, Goi\^ania,
GO \ 74.055-110, Brazil. E-mail: { rdiazmillan@gmail.com}}\footnote{IME, Federal University of Goi\'as, Campus II - 74690-900 - Goi\^ania, GO - Brazil}
and Hung M. Phan\footnote{Department of Mathematical Sciences, Kennedy College of Sciences,
University of Massachusetts Lowell, Lowell, MA \ 01854, USA. E-mail: {hung\_phan@uml.edu} }} \end{center}

\vspace{2mm}

\footnotesize{
\noindent\begin{minipage}{14cm}
{\bf Abstract:}
In this paper, we generalize the classical extragradient algorithm for solving variational inequality problems by utilizing nonzero normal vectors of the feasible set. In particular, conceptual algorithms are proposed with two different linesearchs. We then establish convergence results for these algorithms under mild assumptions. Our study suggests that nonzero normal vectors may significantly improve convergence if chosen appropriately.
\end{minipage}
 \\[5mm]


\hbox to14cm{\hrulefill}\par


\section{Introduction}
In this work, we present conditional extragradient algorithms for solving generally
constrained variational inequality problems by using nonzero normal vectors of the feasible set. Let
$T:\dom(T)\subseteq\RR^n\to\RR^n$ be an operator and let $C \subset \dom(T)$ be a nonempty closed and convex
set, the classical variational inequality problem is formulated as
  \begin{equation}\label{prob}
 \mbox{find} \ \ x_*\in C  \ \  \mbox{such that} \ \ \langle T(x_*), x-x_*\rangle\geq 0,  \ \  \forall \,  x\in C.
   \end{equation}
This problem unifies a broad range of optimization problems and serves as a useful computational framework in very diverse applications. Indeed, \eqref{prob} has been well studied and has numerous important applications in physics, engineering, economics and optimization theory, see, e.g., \cite{hart-stamp, stamp-kinder, 191} and the references therein.

It is well-known that \eqref{prob} is closely related with the so-called {\em dual} problem of the variational inequalities, written as
 \begin{equation}\label{dual}
  \mbox{find} \ \ x_*\in C  \ \  \mbox{such that} \ \   \langle T(x),x-x_*\rangle\ge 0, \ \    \ \  \forall \,  x\in C.
 \end{equation}
We denote the solution set of \eqref{prob} and \eqref{dual} by $S_*$ and $\Sdual$, respectively. Throughout, our standing assumptions are the following:
\begin{enumerate}[leftmargin=0.4in, label=({\bf A\arabic*})]
\item\label{a1} $T$ is continuous on $C$.

\item\label{a2} Problem \eqref{prob} has at-least one solution and all solutions of \eqref{prob} solve the dual problem (\ref{dual}).
\end{enumerate}
Note that assumption \ref{a1} implies $\Sdual\subseteq S_*$ (see Fact~\ref{by-cont} below). So, the existence of solutions of \eqref{dual} implies that of \eqref{prob}. However, the reverse assertion needs generalized monotonicity assumptions.
For example, if $T$ is pseudomonotone then $S_*\subseteq \Sdual$ (see \cite[Lemma~1]{konnov1}).
With this results, we note that \ref{a2} is strictly weaker than pseudomonotonicity of $T$ (see \cite[Example 1.1.3]{konnov-book} and Example~\ref{ex:1} below).
Moreover, the assumptions $S_*\neq \varnothing$ and the continuity of $T$ are natural and classical for most of methods that solve \eqref{prob} in the literature.
Assumption \ref{a2} has also been used in various algorithms for solving \eqref{prob} (see, e.g.,  \cite{konnov-97, konnov1}).

\subsection{Extragradient Algorithm}

Using projection-type algorithms is a popular approach for solving variational inequalities. Excellent surveys on this topic can be found in \cite{pang,konnov-book,harker}. One of the most studied algorithms is the so-called
\emph{extragradient algorithm}, which was first appeared in \cite{kor}. For solving \eqref{prob}, projection methods have to perform at least two projections onto the feasible region at each iteration, because the natural
extension of the projected gradient method
(just one projection when $T=\nabla f$) fails in general for monotone operators (see, e.g., \cite{yu-iu}).
Thus, an extra projection is necessary in order to establish the convergence. A general extragradient scheme can be formulated as follows.

\begin{center}
\fbox{
\begin{minipage}[b]{\textwidth}
\begin{alg}[Extragradient Algorithm]\label{Extragradient} Given $\alpha_k,\beta_k,\gamma_k>0$.
\begin{retraitsimple}
\item[] {\bf Step~0 (Initialization):} Take $x^0\in C$.

\item[] {\bf Step~1 (Iterative Step):} Compute
\begin{subequations}\label{iter1-0}
\begin{align}
z^k&=x^k-\beta_k T(x^k),\label{iter1}\\
y^k&=\alpha_k P_C(z^k) + (1-\alpha_k) x^k,\label{iter1-2}\\
\text{and}\quad x^{k+1}&=P_C\big(x^k-\gamma_k T(y^k)\big)\label{iter1-3}.
\end{align}
\end{subequations}
\item[]{\bf Step~2 (Stopping Test):} If $x^{k+1}=x^k$, then stop. Otherwise, set $k\leftarrow k+1$ and go to {\bf Step~1}.
\end{retraitsimple}
\end{alg}\end{minipage}
}
\end{center}
Next, we describe some strategies to choose the parameters $\alpha_k$, $\beta_k$ and $\gamma_k$ in \eqref{iter1-0} (see, e.g., \cite{pang,konnov-book}).

\noindent ({\bf a}) Constant stepsizes: For each $k$, take $\beta_k=\gamma_k$ where
$0<\check{\beta}\le\beta_k\le \hat{\beta}<+\infty$ and $\alpha_k=1$.

\noindent ({\bf b}) Armijo-type linesearch on the boundary of the feasible set: Set $\sigma>0$, and $\delta\in (0,1)$. For each $k$,
take $\alpha_k=1$ and $\beta_k=\sigma2^{-j(k)}$ where
\begin{equation}\label{E:j(k)}
\left\{\begin{aligned}
&j(k):=\min\left\{\,\,j\in \NN:\| T(x^k)-T(P_C(z^{k,j}))\|\le
\frac{\delta}{\sigma2^{-j}}\,\|x^k-P_C(z^{k,j})\|^2\,\,\right\},\\
&\text{and}\quad z^{k,j}=x^k-\sigma2^{-j} T(x^k).
\end{aligned}\right.
\end{equation}
In this approach, we
take  $\gamma_k=\beta_k$.

\noindent ({\bf c}) Armijo-type linesearch along the feasible direction: Set $\delta\in(0,1)$. For each $k$, take $0<\check{\beta}\le \beta_k\le \hat{\beta}<+\infty$, and $\alpha_k=2^{-\ell(k)}$ where
\begin{equation}\label{E:ell(k)}
\left\{\begin{aligned}
&\ell(k):=\min\left\{\ell\in \NN:\langle T(z^{k,\ell}
),x^k-P_C(z^k)\rangle \geq
\frac{\delta}{\beta_k}\|x^k-P_C(z^k)\|^2\right\},\\
&\text{and}\quad z^{k,\ell}=2^{-\ell}P_C(z^k)+(1-2^{-\ell})x^k.
\end{aligned}\right.
\end{equation}
Then, define 
$\gamma_k=\disp\frac{\langle T(y^k),x^k-y^k\rangle}{\|T(y^k)\|^2}$.

\noindent We provide several comments to explain the differences between these strategies.

Strategy ({\bf a}) was added to the extragradient algorithm in \cite{kor} and it is effective if $T$ is monotone and globally Lipschitz continuous.
The main difficulty of this strategy is the necessity of choosing $\beta_k$ in \eqref{iter1} satisfying $0<\beta_k\le \beta< 1/L$ where the possibly unknown $L$ is the Lipschitz constant of $T$; therefore, the stepsizes should be sufficiently small to ensures the convergence. 

Strategy ({\bf b}) was first studied in \cite{kho} under monotonicity and Lipschitz continuity of $T$. The Lipschitz continuity assumption was removed later
in \cite{IUSEM} by using feasible lineasearch. Note that this strategy requires computing the projection onto $C$ inside the inner loop of the Armijo-type linesearch \eqref{E:j(k)}.
Thus, the need to compute possible  many projections at each iteration $k$ makes Strategy ({\bf b}) inefficient when an explicit formula for $P_C$ is not available. 

Strategy ({\bf c}) was presented in \cite{IS} which demands only one projection for each outer step $k$. This approach guarantees convergence by assuming only the monotonicity of $T$ and the existence of solutions of \eqref{prob}, but not the Lipschitz continuity of $T$.

In Strategies ({\bf b}) and ({\bf c}), the operator $T$ and the projection $P_C$ are evaluated at least twice per iteration. The resulting algorithm is applicable to the whole class of monotone variational inequalities. It has the advantage of not requiring exogenous parameters.
Furthermore, both strategies occasionally allow
long stepsizes by exploiting the information
available at each iteration. 

Extragradient-type algorithms is
currently a subject of intense research (see, e.g.,
\cite{aus-teb, sol-sv-99, sol-tseng, gilabi-2, yu-iu,BI, yu-re-2}). Another variant of Strategy ({\bf c}) was presented in \cite{konnov-97} where the monotonicity was replaced by \ref{a2}. The main difference is that, instead of \eqref{E:ell(k)}, the scheme presented in \cite{konnov-97} performs
\begin{equation}\label{*}
\left\{\begin{aligned}
&\ell(k):=\min\left\{\ell\in \NN:\langle T(z^{k,\ell}),x^k-P_C(z^k)\rangle \geq \delta\langle T(
x^k),x^k-P_C(z^k)\rangle \right\},\\
&\text{and}\quad z^{k,\ell}=2^{-\ell}P_C(z^k)+(1-2^{-\ell})x^k,
\end{aligned}\right.
\end{equation} where $\delta\in(0,1)$.

\subsection{Proposed Schemes}

The paper studies two conceptual algorithms, each of which has three variants. Convergence analysis for both algorithms is established assuming weaker assumptions than previous work \cite{cond,yun-reinier-1}. 
Our scheme was inspired by {\bf Algorithm \ref{Extragradient}} and the conditional subgradient method which was studied in \cite{ds78} and further developed in \cite{cond, ds81}.

Basically, each conceptual algorithm contains a linesearch step and a projection step. First, the linesearch step allows to find a suitable halfspace separating the current iteration and the solution set. We will consider two different linesearches: one on the boundary of the feasible set and one along a feasible direction. Second, the projection step has three variants with different and interesting features on the generated sequence. We also note that some of the proposed variants are related to \cite{BI, IS, sol-sv-99}.
An essential characteristic of the conceptual algorithms is the convergence under very mild assumptions, like the continuity of the operator $T$ (see \ref{a1}),
the existence of solutions of \eqref{prob}, which also solve the dual variational inequality \eqref{dual} (see \ref{a2}).
We would like to emphasize that \ref{a2} is less restrictive than pseudomonotonicity of $T$ and plays a central role in our convergence analysis.

The remaining of the paper is organized as follows. Section~\ref{prel} provides notations and preliminary results,
in which we also prove the convergence of a natural extension of {\bf Algorithm~\ref{Extragradient}} with nonzero normal vectors. The convergence analysis of our conceptual algorithms together with two linesearches is given in Sections \ref{sec-4} and \ref{sec-5}. In Section \ref{s:expl}, we present an example showing that our suggested approach may perform better than previous classical variants. Finally, some concluding remarks are given in Section~\ref{sec-7}.

\section{Preliminaries}
\label{prel}
We begin with some basic notation and definitions, which are standard and follow \cite{librobauch}. Throughout, we write $p:=q$ to indicate that $p$ is defined by $q$. The inner product and the induced norm in $\RR^n$ are denoted respectively by $\langle\cdot,\cdot\rangle$ and
$\|\cdot\|$. We denote the nonnegative integers by $\NN:=\{0, 1, 2,\ldots\}$ and the extended-real line by $\overline{\RR} := \RR \cup\{+\infty\}$. The closed ball centered at $x\in\RR^n$ with radius $\rho>0$ will be denoted by
$\mathbb{B}[x,\rho]:=\{y\in\RR^n\colon\|y-x\|\leq\rho\}$. The domain of a function $f:\RR^n\rightarrow\overline{\RR}$ is
defined by $\dom(f):=\{x\in\RR^n : f(x)<+\infty\}$ and we say that $f$ is proper if $\dom(f)\neq \varnothing$. For any set $G$,
{\rm cl}(G) and ${\rm cone}(G)$ respectively denote the
topological closure and the conic hull of $G$. Finally, let $T: \RR^n\rightrightarrows\RR^n$ be an operator. Then, the domain and the graph of $T$ are given by
$\dom(T) := \{x\in \RR^n : T(x)\neq \varnothing\}$ and ${\rm Gph}(T):=\{(x,u)\in\RR^n\times\RR^n : u\in T(x)\}$.

\begin{definition}[normal cone]
Let $C$ be a subset of $\RR^n$ and let $x\in C$. A vector $u\in\RR^n$ is called a {\em normal} to $C$ at $x$
if for all $y\in C$, $\scal{u}{y-x}\leq 0$.
The collection of all such normal $u$ is called the \emph{normal cone} of $C$ at $x$ and is denoted by $\mcN_C(x)$. If $x\notin C$, we define $\mcN_C(x)=\varnothing$.
\end{definition}
In some special cases, formulas for normal cone can be obtained explicitly, for example, polyhedral sets \cite{cond}, closed convex cones \cite[Example~2.62]{bonnans}, sets defined by smooth functional constraints \cite[Theorem~6.14]{rock-VA} (see also \cite[Theorem~23.7]{rock-CA} and \cite[Proposition~2.61]{bonnans}).

The normal cone can be seen as an operator, i.e., $\mcN_C:
C\subset \RR^n\rightrightarrows\RR^n:x\mapsto
\mcN_C(x)$. Recall that the indicator function of $C$ is defined by $\delta_C(y):=0$, if $y\in C$ and $+\infty$, otherwise, and the classical convex subdifferential operator for a proper function $f:\RR^n\rightarrow\overline{\RR}$ is defined by $ \partial f:\RR^n\rightrightarrows\RR^n:
x\mapsto \partial f(x):=\left\{u\in \RR^n: f(y)\geq f(x)+\langle u,y-x\rangle ,\; \forall\, y\in \RR^n\right\}$.
Then, it is well-known that the normal cone operator can be expressed as $\mcN_C=\partial \delta_C$.

\begin{fact}{\rm(See~\cite[Proposition~4.2.1(ii)]{iusem-regina})}
\label{nor-cone}
The normal cone operator for $C$, $\mcN_C$, is a maximal
monotone operator and its graph, ${\rm Gph}(\mcN_C)$, is closed,
i.e., for every sequence $(x^k,u^k)_{k\in \NN}\subset
{\rm Gph}(\mcN_C)$ that converges to some $(x,u)$, we have $(x,u)\in
{\rm Gph}(\mcN_C)$.
\end{fact}

Next, recall that the orthogonal projection of $x$ onto $C$, $P_C(x)$, is the unique point in $C$ such that $\| P_C(x)-x\| \le \|x-y\|$ for all $y\in C$. Some well-known facts about orthogonal projections are presented below.
\begin{fact}\label{proj}
For all $x,y\in \RR^n$ and all $z\in C $, the following hold:
\begin{enumerate}
\item\label{proj-i} $\|P_C(x)-P_C(y)\|^2 \leq \|x-y\|^2-\|(x-P_C(x))-\big(y-P_C(y)\big)\|^2$ (a.k.a. firm nonexpansiveness).
\item\label{proj-ii} $\langle x-P_C(x),z-P_C(x)\rangle \leq 0.$
\item\label{proj-iii} Let $x\in C$, $y\in\RR^n$ and $z=P_C(y)$, then $\langle x-y,x-z\rangle \geq \|x-z\|^2$.
\end{enumerate}
\end{fact}
\proof
\ref{proj-i} \& \ref{proj-ii}: See \cite[Lemmas~1.1 and 1.2]{zarantonelo}.

\noindent \ref{proj-iii}: Using \ref{proj-ii}, we have $\scal{x-y}{x-z}=\scal{x-z}{x-z}+\scal{x-z}{z-y}\geq\|x-z\|^2$.
\endproof
\begin{corollary}\label{projP1}
For all $x,p\in \RR^n$ and $\alpha>0$, we have
\[\disp
\frac{x-P_C(x-\alpha p)}{\alpha}\in p+\mcN_C(P_C(x-\alpha
p)).\]
\end{corollary}
\proof
Let $z=x-\alpha p$, then the conclusion follows from $z-P_C(z)\in\mcN_C(P_C(z))$.
\endproof

Next, we present some lemmas that are useful in the sequel.
\begin{lemma}\label{setprop}
Let $H\subseteq \RR^n$ be a closed halfspace and $C\subseteq \RR^n$ such that $H\cap C\neq \emp$. Then, for every $x\in C$, we have $
P_{H\cap C}(x)=P_{H\cap C}(P_{H}(x))$.
\end{lemma}
\proof
If $x\in H$, then $x=P_{H\cap C}(x)=P_{H\cap C}(P_{H}(x))$. Suppose that $x\notin H$. Fix any $y\in C\cap H$.
Since $x\in C$ but $x\notin H$, there exists $\gamma \in [0,1)$, such that $\tx=\gamma x+(1-\gamma)y\in C\cap \bdry H$,
where $\bdry H$ is the hyperplane boundary of $H$. Hence, $(\tx-P_H(x))\bot( x-P_H(x))$ and $(P_{H\cap C}(x)-P_H(x))\bot( x-P_H(x))$, then
\begin{equation}\label{P1}
\|\tx-x\|^2
=\|\tx-P_H(x)\|^2+\|x-P_H(x)\|^2,
\end{equation}
and
\begin{equation}\label{P2}
\|P_{H\cap C}(x)-x\|^2
=\|P_{H\cap C}(x)-P_H(x)\|^2+\|x-P_H(x)\|^2,
\end{equation}
respectively. Using \eqref{P1} and \eqref{P2}, we get
\begin{align*}
&\|y-P_{H}(x)\|^2\geq\|\tx-x\|^2
=\|\tx-P_H(x)\|^2+\|P_H(x)-x\|^2
\geq\|\tx-P_{H}(x)\|^2.\\
&=\|\tilde{x}-x\|^2-\|x-P_{H}(x)\|^2
\geq\, \|P_{H\cap C}(x)-x\|^2-\|x-P_{H}(x)\|^2
= \, \|P_{H\cap C}(x)-P_{H}(x)\|^2.
\end{align*}
So,
$\|y-P_{H}(x)\|\geq \|P_{H\cap C}(x)-P_{H}(x)\|$ for all $y\in C\cap H$. Thus, $P_{H\cap C}(x)=P_{C\cap
H}(P_{H}(x))$.
\endproof

\begin{lemma}\label{l:lim-2} Let $S$ be a nonempty, closed and convex set. Let $x^0,x\in\RR^n$. Assume that $x^0\notin S$ and that
$S\subseteq W(x)= \{y\in \RR^n :\langle y-x,x^0-x\rangle \le0\}$. Then, $
x\in B[\tfrac{1}{2}(x^0+\ox),\tfrac{1}{2}\rho]$, where
$\ox=P_{S}(x^0)$ and $\rho={\rm dist}(x^0, S):=\|x_0-P_{S}(x_0)\|$.
\end{lemma}
\proof Since $S$ is convex and closed, $\ox=P_{S}(x^0)$ and
$\rho={\rm dist}(x^0,S)$ are well-defined. $S \subseteq W(x)$
implies that $\ox=P_{S}(x^0)\in W(x)$. Define
$v:=\tfrac{1}{2}(x_0+\ox)$ and $r:=x^0-v=\tfrac{1}{2}(x^0-\ox)$,
then $\ox-v=-r$ and
$\|r\|=\tfrac{1}{2}\|x^0-\ox\|=\tfrac{1}{2}\rho$. It follows that
\begin{align*}
0&\geq\langle \ox-x,x^0-x\rangle =\scal{\ox-v+v-x}{x^0-v+v-x}\\
&=\scal{-r+(v-x)}{r+(v-x)}=\|v-x\|^2-\|r\|^2.
\end{align*}
So, $x\in B[v,r]$ and the proof is complete.
\endproof

\begin{definition}[Fej\'er convergence]
Let $S$ be a nonempty subset of $\RR^n$. A sequence $(x^k)_{k\in\NN}\subset \RR^n$ is said to be Fej\'er convergent to $S$ if and only if for all
$x\in S$ there exists $k_0\in \NN$ such that
$\|x^{k+1}-x\| \le \|x^k - x\|$ for all $k\ge k_0$.
\end{definition}
Fej\'er convergence was introduced in \cite{browder} and has been elaborated further in \cite{IST,borw-baus}. The following are useful properties of Fej\'er sequences.
\begin{fact}\label{punto}
If $(x^k)_{k\in \NN}$ is Fej\'er convergent to $S$, then the following hold
\begin{enumerate}
\item\label{punto-i} The sequence $(x^k)_{k\in \NN}$ is bounded.
\item\label{punto-ii} The sequence $\big (\|x^k-x\|\big)_{k \in \NN}$ converges for all $x\in S$.
\item\label{punto-iii} If an accumulation point $x_{*}$ belongs to $S$, then the sequence $(x^k)_{k\in \NN}$ converges to $x_{*}$.
\end{enumerate}
\end{fact}
\proof
\ref{punto-i} and \ref{punto-ii}: See \cite[Proposition~5.4]{librobauch}. \ref{punto-iii}: See \cite[Theorem~5.5]{librobauch}.
\endproof

\noindent We recall the following well-known characterization of $S_*$ which will be used repeatedly.
\begin{fact}
{\rm(See~\cite[Proposition 1.5.8]{pang})}
\label{f:0822b}
The following are equivalent:
\begin{enumerate}
\item $x\in S_*$.
\item $-T(x)\in\mcN_C(x)$.
\item For all $\beta>0$, we have $x=P_C(x-\beta T(x))$.
\end{enumerate}
\end{fact}

\begin{proposition}\label{parada}
	Given $T: \dom(T)\subseteq \RR^n\to \RR^n$ and $\alpha>0$. If
	$x=P_C(x-\alpha (T(x)+ u))$ for some $u\in \mcN_C(x)$,
	then $x\in S_*$, or equivalently, $x=P_C(x-\beta T(x))$ for all $\beta>0$.
\end{proposition}
\proof
It follows from Corollary \ref{projP1} that $0\in T(x)+ u+
\mcN_C(x),$ which implies that $-T(x)\in \mcN_C(x)$. The conclusion is now immediate from Fact~\ref{f:0822b}.
\endproof
\begin{remark}
It is quite easy to see that the reverse of Proposition~\ref{parada} is not true in general.
\end{remark}
The next result will be used to prove that all accumulation  points of the sequences generated by the proposed algorithms belong to the solution set of problem \eqref{prob}.
\begin{fact}
{\rm(See~\cite[Lemma~3]{yun-iusem-2012})}
\label{by-cont}
If $T:\dom(T)\subseteq\RR^n\rightarrow\RR^n$ is continuous, then
$\Sdual\subseteq S_*$.
\end{fact}

\begin{lemma}\label{propseq}
For any $(z,v)\in {\rm Gph}(\mcN_C)$ define $H(z,v) := \big\{ y\in \RR^n : \la
T(z)+v,y-z\rangle \le 0\big \}$.
Then, $S_*= \Sdual\subseteq H(z,v)$.
\end{lemma}
\proof $S_*=\Sdual$ by Assumption \ref{a2} and Fact~\ref{by-cont}.
Take $x_{*}\in \Sdual$, then  $\langle T(z),x_*-z\rangle \leq 0$  for all $z\in
C$. Since $(z,v)\in {\rm Gph}(\mcN_C)$, we have $\langle v,
x_{*}-z\rangle \le 0$. Summing up these inequalities, we get $\langle T(z)+v, x_{*}-z\rangle \le 0$.
Then, $x_{*}\in H(z,v)$.
\endproof

In view of Lemma~\ref{propseq}, Assumptions \ref{a1} and \ref{a2} imply that $\Sdual=S_*$. Hence, the next result is immediate.

\begin{lemma}\label{sol-convex-closed} If
$T:\dom(T)\subseteq\RR^n\rightarrow\RR^n$ is continuous and Assumption {\rm \ref{a2}} holds, then $S_*$ is a closed and convex set.
\end{lemma}

\subsection{Extragradient Algorithm with Normal Vectors}

We now show that it is possible to incorporate normal vectors of the feasible sets into the extragradient algorithm. As we will see below, this approach generalizes {\bf Algorithm~\ref{Extragradient}} with Strategy ({\bf a}). To proceed, we assume that $T$ is Lipschitz with constant $L$ and \ref{a2} holds.
\vspace*{-.15in}
\begin{center}\fbox{\begin{minipage}[b]{\textwidth}
\begin{alg}[Extragradient
Algorithm with Normal Vectors]\label{Cond-Ext} Take $(\beta_k)_{k\in
\NN}\subset[\check{\beta},\hat{\beta}]$ such that
$0<\check{\beta}\le \hat{\beta}<1/(L+1)$ and $\delta\in (0,1)$.
\begin{retraitsimple}
\item[] {\bf Step~0 (Initialization):} Take $x^0\in C$ and set $k\leftarrow0$.

\item[] {\bf Step~1 (Stopping Test):} If $x^k=P_C(x^k-\beta_k T(x^k))$, then stop. Otherwise:

\item[] {\bf Step~2 (First Projection):} Take $u^k\in \mcN_C(x^k)$ such that
\begin{align}
\|u^k\|&\leq \delta \|x^k-P_C(x^k-\beta_k (T(x^k)+u^k))\|,\label{uk-rk}\\
z^k&=P_C(x^k-\beta_k( T(x^k)+u^k)).\label{iter1-C}
\end{align}
\item[] {\bf Step~3 (Second Projection):} Take $v^k\in \mcN_C(z^k)$
such that
\begin{equation}\label{vk-uk}
\|v^k-u^k\|\le \|x^k-z^k\|.
\end{equation}
Set
\begin{equation}\label{iter3-C}
x^{k+1}=P_C(x^k-\beta_k( T(z^k)+ v^k)).
\end{equation}
Set $k\leftarrow k+1$ and go to {\bf Step~1}.
\end{retraitsimple}
\end{alg}\end{minipage}}\end{center}

\begin{proposition}\label{ukyvk-ok}
{\bf Algorithm \ref{Cond-Ext}} is well-defined.
\end{proposition}
\proof
It is sufficient to prove that if Step~1 is not satisfied, i.e.,
\begin{equation}\label{e:0822a}
\|x^k-P_C(x^k-\beta_k T(x^k))\|>0.
\end{equation}
then Steps~2 and 3 are attainable.

\noindent {\em Step~2 is attainable:} Suppose that \eqref{uk-rk} does not hold for every  $\alpha u^k\in
\mcN_C(x^k)$ with $\alpha>0$, i.e., $\|\alpha u^k\|> \delta \|x^k-P_C(x^k-\beta_k (T(x^k)+\alpha
u^k))\|\ge 0.$
 Taking limit when $\alpha$ goes to $0$, we get
$\|x^k-P_C(x^k-\beta_k T(x^k))\|=0$, which contradicts \eqref{e:0822a}.

\noindent {\em Step~3 is attainable:} Suppose that \eqref{vk-uk} does not hold for every $\alpha v^k\in \mcN_C(z^k)$ with $\alpha>0$, i.e., $ \|\alpha v^k-u^k\|> \|x^k-z^k\|,$
where $z^k=P_C(x^k-\beta_k(T(x^k)+u^k))$ as \eqref{iter1-C} and $u^k\in \mcN_C(x^k)$ satisfying \eqref{uk-rk}.
Letting $\alpha$ goes to $0$ and using \eqref{uk-rk}, we get
$
\|x^k-z^k\|\le \|u^k\|\le \delta \|x^k-z^k\|.
$
So, $x^k=z^k$. Then, Proposition \ref{parada} implies a contradiction to \eqref{e:0822a}.
\endproof



It is immediate from Proposition~\ref{parada} that if the Stopping Test is satisfied for $x_k$, then $x^k\in S_*$. So we investigate the remaining case that the Stopping Test is not satisfied for all $x^k$. In this case, we will prove that the algorithm generates an infinite sequence $(x^k)_\kkk$ that converges to $S_*$.

\begin{lemma}\label{des-fejer-C}
Suppose that $T$ is Lipschitz continuous with constant $L$. Let
$x_*\in S_*$. Suppose also that Stopping Test is not satisfied for $x^k$. Then Step~4 generates $x^{k+1}$ and that
\begin{equation*}
\|x^{k+1}-x_*\|^2\le
\|x^{k}-x_*\|^2-(1-\beta_k^2(L+1)^2)\|z^k-x^k\|^2.
\end{equation*}
\end{lemma}
\proof Define $w^k= x^k-\beta_k( T(z^k)+ v^k)$ with $v^k\in
\mcN_C(z^k)$ taken from Step~3. Then, using \eqref{iter3-C} and applying Proposition~\ref{proj}(i), with $x=w^k$ and $y=x_*$, we get
\begin{align}\label{ine-1}\nonumber
\|x^{k+1}-x_*\|^2
&\leq \|w^k-x_*\|^2-\|w^k-P_C(w^k)\|^2\\\nonumber
&\leq \|x^k-x_*-\beta_k( T(z^k)+
v^k)\|^2-\|x^k-x^{k+1}-\beta_k(
T(z^k)+ v^k)\|^2\\
&= \|x^k-x_*\|^2-\|x^k-x^{k+1}\|^2+2\beta_k\langle T(z^k)+ v^k,
x_*-x^{k+1} \ra.
\end{align}
Since $v^k\in \mcN_C(z^k)$ and \ref{a2}, we have
\begin{align*}
\langle T(z^k)+ v^k, x_*-x^{k+1} \rangle =& \langle T(z^k)+ v^k, z^k-x^{k+1} \ra
+\langle T(z^k)+ v^k, x_*-z^k\rangle \\\le &\langle T(z^k)+ v^k, z^k-x^{k+1}
\ra+\langle T(z^k), x_*-z^k\rangle \\ \le &\langle T(z^k)+ v^k, z^k-x^{k+1}\ra.
\end{align*}
Substituting into \eqref{ine-1} yields
\begin{align}\label{ine-2}\nonumber
\|x^{k+1}-x_*\|^2 \le&  \|x^k-x_*\|^2-\|x^k-x^{k+1}\|^2-2\beta_k\la
T(z^k)+ v^k, x^{k+1}-z^k\rangle \\
=& \|x^k-x_*\|^2-\|x^k-z^k\|^2-\|z^k-x^{k+1}\|^2\nonumber \\+&2\la
x^k-\beta_k(T(z^k)+ v^k)-z^k, x^{k+1}-z^k\ra.
\end{align} Define $\ox^k=x^k-\beta_k(T(x^k)+u^k)$ with $u^k\in\mcN_C(x^k)$ taken from Step~2 and recall that $z^k=P_C(\bar{x}^k)$ and that
$x^{k+1}=P_C(w^k)=P_C(x^k-\beta_k(T(z^k)+
v^k))$, we have
\begin{align}\label{ine-3}\nonumber
2 \langle x^k-&\beta_k(T(z^k)+ v^k)-z^k, x^{k+1}-z^k\rangle \\
\nonumber &=2 \langle w^k-P_C(\ox^k), P_C(w^k)-P_C(\ox^k)\rangle \\
\nonumber &=2\la
\ox^k-P_C(\ox^k), P_C(w^k)-P_C(\ox^k)\rangle + 2\la
w^k-\ox^k, P_C(w^k)-P_C(\ox^k)\rangle \\\nonumber &\le 2 \la
w^k-\ox^k, P_C(w^k)-P_C(\ox^k)\rangle \\\nonumber &= 2 \la
w^k-\ox^k, x^{k+1}-z^k\rangle = 2 \beta_k \la
(T(x^k)+u^k)-(T(z^k)+v^k), x^{k+1}-z^k\rangle \\\nonumber &\le
2\beta_k\left(\|T(z^k)-T(x^k)\|+\|v^k-u^k\|\right)\|x^{k+1}-z^k\|\\&\le
2\beta_k (L+1)\|z^k-x^k\|\|x^{k+1}-z^k\|\le \beta_k^2
(L+1)^2\|z^k-x^k\|^2 + \|x^{k+1}-z^k\|^2,
\end{align}
using Proposition \ref{proj}(ii), with $x=x^k-\beta_k(T(x^k)+ u^k)$
and $z=x^{k+1}$, in the first inequality, the Cauchy-Schwarz
inequality in the second one and the Lipschitz continuity of $T$ and
\eqref{vk-uk} in the third one. Finally, the conclusion follows from \eqref{ine-3} and
\eqref{ine-2}.
\endproof

\begin{corollary}\label{fejer-C}
The sequence $(x^k)_{k\in \NN}$ is Fej\'er convergent to $S_*$ and $\disp\lim_{k\rightarrow\infty} \|z^k-x^k\|=0$.
\end{corollary}
\proof
It follows from Lemma~\ref{des-fejer-C} and
$\beta_k\le \hat{\beta}< 1/(L+1)$ that
\begin{equation*}\label{des-Lip-1}
\|x^{k+1}-x_*\|^2\leq \|x^{k}-x_*\|^2-(1-\hat{\beta}^2L^2)\|z^k-x^k\|^2\leq \|x^{k}-x_*\|^2.
\end{equation*}
So, $(x^k)_\kkk$ is Fej\'er convergent to $S_*$. Now Fact~\ref{punto}(ii) together with the above inequality imply $\disp\lim_{k\rightarrow\infty} \|z^k-x^k\|=0$.
\endproof

\begin{proposition}
The sequence $(x^k)_{k\in \NN}$ converges to a point in $S_*$.
\end{proposition}
\proof
The sequence $(x^k)_{k\in \NN}$  is bounded by Lemma
\ref{des-fejer-C} and Fact~\ref{punto}(i). Let $\tilde{x}$ be an accumulation point of some subsequence $(x^{i_k})_{k\in\NN}$.
By Corollary~\ref{fejer-C}, $\tilde{x}$ is also an accumulation point of $(z^{i_k})_{k\in\NN}$. Without loss of generality,
we suppose that the corresponding parameters $(\beta_{i_k})_{k\in\NN}$ and $(u^{i_k})_{k\in\NN}$ converge to $\tilde{\beta}$ and $\tilde{u}$,
respectively. Since $z^k=P_C(x^k-\beta_k( T(x^k)+u^k))$, taking the limit along
the subsequence $(i_k)_{k\in\NN}$, we obtain
$
\tilde{x}=P_C(\tilde{x}-\tilde{\beta}( T(\tilde{x})+\tilde{u})).
$
Therefore, Fact \ref{nor-cone} and Proposition \ref{parada} imply $\tilde{x}\in S_*$. Finally, we apply Fact~\ref{punto}\ref{punto-iii}.
\endproof

\section{Conceptual Algorithm with Linesearch \ref{boundary}}\label{sec-4}

In this section, we study a conceptual algorithm, in which we use a linesearch along the boundary of the feasible set to obtain the stepsizes. Indeed, {\bf Linesearch~\ref{boundary}} given below generalizes Strategies ({\bf b}) by involving normal vectors to feasible sets.
\vspace*{-.15in}
\begin{center}\fbox{\begin{minipage}[b]{\textwidth}
\begin{linesr}{B}
{\rm(Linesearch on the boundary)}
\label{boundary}

\medskip
{\bf Input:} $(x,u,\sigma,\delta,M)$. Where $x\in C$, $u\in \mcN_C(x)$, $\sigma>0$, $\dd\in(0,1)$, and $M>0$.

Set $\alpha=\sigma$ and $\theta\in (0,1)$ and choose $u\in \mcN_C(x)$. Denote $z_{\alpha}=P_C(x-\alpha(T(x)+\alpha u))$ and choose
$v_\alpha\in\mcN_C(z_\alpha)$ with $\|v_\alpha\|\leq M$.

\begin{retraitsimple}
\item[] {\bf While} $\alpha\|T(z_{\alpha})-T(x)+\alpha v_{\alpha}-\alpha u\| > \delta\|z_{\alpha}-x\|$   {\bf do}

$\alpha\leftarrow\theta \alpha$\quad and\quad choose any $v_{\alpha}\in\mcN_C(z_{\alpha})$ with $\|v_\alpha\|\leq M$.

\item[] {\bf End While}
\end{retraitsimple}
{\bf Output:} $(\alpha,z_\alpha,v_{\alpha})$.
\end{linesr}\end{minipage}}\end{center}
We now show that {\bf Linesearch
\ref{boundary}} is well-defined assuming only \ref{a1}, i.e., continuity of $T$.
\begin{lemma}\label{boundary-well} If $x\in C$ and $x\notin S_*$, then {\bf Linesearch \ref{boundary}} stops after finitely many steps.
\end{lemma}
\proof
Suppose on the contrary that {\bf Linesearch~\ref{boundary}} does not stop for all
$\alpha\in \mcP:=\{\sigma, \sigma\theta, \sigma\theta^2, \ldots \}$ and the chosen vectors
\begin{equation}
v_{\alpha}\in\mcN_C(z_\alpha),\quad \|v_\alpha\|\leq M,\quad
z_\alpha=P_C(x-\alpha (T(x)+\alpha u)).\label{vz_alpha*}
\end{equation}
We have
\begin{equation}\label{no-armijo}
\alpha\|T(z_\alpha)-T(x)+\alpha v_{\alpha}-\alpha u\|
>\delta \|z_\alpha-x\|.
\end{equation}
Next, divide both sides of \eqref{no-armijo} by $\alpha>0$ and let $\alpha$ goes to $0$. Due to the boundedness of
$(v_\alpha)_{\alpha\in\mcP}$ and the continuity of $T$, we obtain
\begin{equation*}
0=\liminf_{\alpha \to 0} \|T(z_\alpha)-T(x)+\alpha v_{\alpha}-\alpha
u\|\ge\liminf_{\alpha \to 0} \frac{\|x-z_\alpha\|}{\alpha}\ge 0.
\end{equation*}
Using $z_\alpha$ in \eqref{vz_alpha*}, we have
\begin{equation}\label{solP1}
\liminf_{\alpha \to 0}\frac{\|x- P_C(x-\alpha (T(x)+\alpha u))\|}{\alpha}= 0.
\end{equation}
On the other hand, Corollary~\ref{projP1} implies
\[\frac{x- P_C(x-\alpha (T(x)+
\alpha u))}{\alpha}\in T(x)+\alpha u+\mcN_C( P_C(x-\alpha
(T(x)+\alpha u))).\]
From  \eqref{solP1}, the continuity of the projection and the closedness of ${\rm Gph}(\mcN_C)$ imply $0\in T(x)+\mcN_C(x)$, which is a contradiction
since $x \notin S_*$.
\endproof

Next, we present the conceptual algorithm, which is related to {\bf Algorithm~\ref{Extragradient}} with Strategy ({\bf b}) when nonzero normal vectors are used. Here, we assume that \ref{a1} and \ref{a2} hold.
\vspace*{-.15in}
\begin{center}
\fbox{\begin{minipage}[b]{\textwidth}
\begin{Calg}{B}
\label{A2} Given  $\sigma>0$, $\delta\in(0,1)$, and $M>0$.

\item[ ]{\bf Step~0 (Initialization):} Take $x^0\in C$ and set $k\leftarrow 0$.

\item[ ]{\bf Step~1 (Stopping Test):}
If $x^k=P_C(x^k-T(x^k))$, then stop. Otherwise,

\item[ ]{\bf Step~2 (Linesearch \ref{boundary}):} Take $u^k\in \mcN_C(x^k)$ with $\|u^k\|\leq M$
and set
$$(\alpha_k, z^k, v^k)= {\bf Linesearch \; \ref{boundary}}\;(x^k, u^k, \sigma,\delta,M),$$
i.e., $(\alpha_k,z^k,v^k)$ satisfies
\begin{subequations}\label{zk212}
\begin{align}
& v^k\in \mcN_C(z^k) \ \mbox{with}\ \|v^k\|\leq M,\quad \alpha_k\leq\sigma,\label{zk212-a}\\
& z^k=P_C(x^{k}-\alpha_k(T(x^{k})+\alpha_k u^k)),\label{zk212-b}\\
&\alpha_k\|T(z^k)-T(x^k)+\alpha_k(v^k-u^k)\|\leq \delta\|z^k-x^k\|.\label{zk212-c}
\end{align}
\end{subequations}
\item[ ]{\bf Step~3 (Projection):} Set \
${\ov}^k:=\alpha_k v^k$ \ and \ $x^{k+1}:=\mcF_B(x^k)$.

\item[ ]{\bf Step~4:} Set $k\leftarrow k+1$ and go to {\bf Step~1}.
\end{Calg}\end{minipage}}\end{center}

\noindent We consider three variants of $\mcF_B$ in Step~3:
\begin{align}
\mcF_{\rm B.1}(x^k) =& P_C\big(P_{H(z^k,\ov^k)}(x^k)\big),\label{P112}  \quad   &{(\bf Variant\; B.1)} \\
\mcF_{\rm B.2}(x^k) =& P_{C\cap H(z^k,{\ov}^k)}(x^k),\label{P122}  \quad   &{(\bf Variant\; B.2)}\\
\mcF_{\rm B.3}(x^k) =& P_{C\cap H(z^k,{\ov}^k)\cap
W(x^k)}(x^0),\label{P132}  \quad   & {(\bf Variant\; B.3)}
\end{align} where
\begin{subequations}\label{e:HW}
\begin{align}
H(z^k,\ov^k)&:=\big\{ y\in \RR^n : \langle T(z^k)+\ov^k,y-z^k\rangle \le 0\big \},\label{H(x,v)}\\
\text{and}\qquad
W(x^k)&:=\big\{ y\in \RR^n : \langle y-x^k,x^0-x^k\rangle \le 0\big \}.\label{W(x)}
\end{align}
\end{subequations}
These halfspaces have been widely used in the literature, see, e.g., \cite{yuniusem,sva,yun-reinier-1} and the references therein.
Our goal is to analyze the convergence of these variants. First, we start by showing that the algorithm is well-defined.
\begin{proposition}\label{propdef*}
Assume that $\mcF_B(x^k)$ is well-defined whenever $x^k$ is available. Then, {\bf Conceptual Algorithm \ref{A2}} is also well-defined.
\end{proposition}
\proof
If the Stopping Test is not satisfied, then Step~2 is attainable by Lemma~\ref{boundary-well}. So the algorithm is well-defined.
\endproof

\begin{proposition}\label{H-separa-x}
$x^k\in S_*$ if and only if $x^k \in H(z^k,{\ov}^k)$, where $z^k$ and ${\ov}^k$ are obtained in Steps~2 and~3, respectively.
\end{proposition}
\proof
If $x^k \in S_*$, then $x^k \in
H(z^k,{\ov}^k)$ by Lemma \ref{propseq}. Now suppose that $x^k\notin S_*$. Define $\bar{u}^k=\alpha_k u^k\in
\mcN_C(x^k)$ and $w^k=x^k-\alpha_k( T(x^k)+\bar{u}^k)$. Then,
\begin{align}\label{paralim}\nonumber
\alpha_k\langle T(z^k)+{\ov}^k,x^k-z^k \ra
&=\alpha_k\langle T(z^k)-T(x^k)+{\ov}^k-\bar{u}^k, x^k-z^k\ra
+\alpha_k \langle T(x^k)+\bar{u}^k,x^k -z^k\rangle \nonumber\\
&=\alpha_k\langle T(z^k)-T(x^k)+{\ov}^k-\bar{u}^k, x^k-z^k\ra
+\langle x^k-w^k,x^k -z^k\rangle \nonumber\\
&\geq
-\alpha_k\|T(z^k)-T(x^k)+{\ov}^k-\bar{u}^k\|\cdot\| x^k-z^k\|+\|x^k-z^k\|^2  \nonumber \\
&\geq-\dd\|x^k-z^k\|^2+\|x^k-z^k\|^2=(1-\delta)\|x^k-z^k \|^2>0,
\end{align}
where we have used {\bf Linesearch \ref{boundary}} and Fact~\ref{proj}\ref{proj-iii} in the second inequality. It follows that $x^k\notin H(z^k,{\ov}^k)$ by the definition of $H(z^k,{\ov}^k)$.
\endproof

Let $(x^k)_{k\in \NN}$, $(z^k)_{k\in \NN}$ and $(\alpha_k)_{k\in \NN}$ be sequences generated by {\bf Conceptual Algorithm \ref{A2}} and suppose that $x^k\notin S_*$. Using \eqref{paralim}, we obtain a useful algebraic property
\begin{equation}\label{d:useful1}
\forall k\in\NN:\quad\langle T(z^{k})+{\ov}^{k},x^{k}-z^{k} \rangle  \geq
\frac{(1-\delta)}{\alpha_k}\|x^k-z^k\|^2.
\end{equation}

\begin{proposition}\label{p:aBdry_xk1_neq_xk}
If Stopping Test is not satisfied at $x^k$, then {\bf Conceptual Algorithm~\ref{A2}} generates $x^{k+1}\neq x^k$.
\end{proposition}
\proof
Suppose on the contrary that $x^{k+1}=x^k$. Consider three cases.

If {\bf Variant~B.1} is used, then $x^{k+1}=P_C\big(P_{H(z^k,{\ov}^k)}(x^k)\big)=x^k$. Then Fact~\ref{proj}\ref{proj-ii} implies
\begin{equation}\label{proyex*}
\langle P_{H(z^k,{\ov}^k)}(x^k)-x^{k}, z-x^{k}\rangle =
\langle P_{H(z^k,{\ov}^k)}(x^k)-x^{k+1}, z-x^{k+1}\rangle \leq 0,
\end{equation} for all $z\in C$. Using again Fact~\ref{proj}(ii),
\begin{equation}\label{proyeh*}
\forall z\in H(z^k,{\ov}^k):\quad\langle P_{H(z^k,{\ov}^k)}(x^k)-x^k, P_{H(z^k,{\ov}^k)}(x^k)-z\ra
\leq 0.
\end{equation}
Note that $z^k\in C\cap H(z^k,{\ov}^k) \neq \emp$. So, setting $z=z^k$ and summing up \eqref{proyex*} and \eqref{proyeh*}, we obtain $\|x^k-P_{H(z^k,{\ov}^k)}(x^k)\|^2=0$.
Hence, $x^k=P_{H(z^k,{\ov}^k)}(x^k)$, i.e., $x^k\in
H(z^k,{\ov}^k)$. 

If {\bf Variant~B.2} is used, then $x^{k+1}=P_{C\cap H(z^k,{\ov}^k)}(x^k)=x^k$. So $x^k\in H(z^k,{\ov}^k)$.

If {\bf Variant~B.3} is used, then $x^{k+1}=P_{C\cap H(z^k,{\ov}^k)\cap W(x^k)}(x^0)=x^k$. So $x^k\in  H(z^k,{\ov}^k)$.

Hence, in all cases, we have showed that $x^k\in H(z^k,{\ov}^k)$, which implies $x^k\in S_*$ by Proposition~\ref{H-separa-x}. By Fact~\ref{f:0822b}, we get $x^k=P_C(x^k-T(x^k))$, i.e., Stopping Test is satisfied at $x^k$, a contradiction.
\endproof

In view of Proposition~\ref{p:aBdry_xk1_neq_xk}, we will only examine the case that Stopping Test is not satisfied for all $x^k$. In this case, {\bf Conceptual Algorithm~\ref{A2}} generates an infinite sequence $(x^k)_{k\in \NN}$ such that $x^k\notin S_*$ for all $\kkk$.

\subsection{Convergence Analysis of Variant B.1}\label{sec-5.1}

We consider the case {\bf Variant
B.1} is used and the algorithm generates an infinite sequence $(x^k)_\kkk$ such that $x^k\notin S_*$ for all $\kkk$. Note that by Lemma~\ref{propseq},
$H(z^k,{\ov}^k)$ is nonempty for all $k$. Thus, the projection step \eqref{P112} is well-defined, so is the whole algorithm.

\begin{proposition}\label{prop2}
The following hold:
\begin{enumerate}
\item\label{prop2-i} The sequence $(x^k)_{k\in \NN}$ is Fej\'er convergent to $S_*$.
\item\label{prop2-ii} The sequence $(x^k)_{k\in \NN}$ is bounded.
\item\label{prop2-iii} $\disp\lim_{k\to \infty}\langle T(z^k)+{\ov}^k,x^k-z^k \rangle =0$.
\end{enumerate}
\end{proposition}
\proof
\ref{prop2-i}: Take $x_{*}\in S_*$. Note that, by definition $(z^k,{\ov}^k) \in {\rm Gph}(\mcN_C)$. Using \eqref{P112},
Fact~\ref{proj}(i) and Lemma \ref{propseq}, we have
\begin{equation}\label{fejer-des}
\begin{aligned}
\|x^{k+1}-x_{*}\|^2&=\|P_{C}(P_{H(z^k,{\ov}^k)}(x^k))-P_{C}(P_{H(z^k,{\ov}^k)}(x_{*}))\|^2\\
&\leq
\|P_{H(z^k,{\ov}^k)}(x^k)-P_{H(z^k,{\ov}^k)}(x_{*})\|^2\\
&\leq \,
\|x^k-x_{*}\|^2-\|P_{H(z^k,{\ov}^k)}(x^k)-x^k\|^2
\leq\|x^k-x_{*}\|^2.
\end{aligned}
\end{equation}

\noindent \ref{prop2-ii}: Follows from \ref{prop2-i} and Fact~\ref{punto}\ref{punto-i}.

\noindent \ref{prop2-iii}: Take $x_{*} \in S_*$ and notice that $
P_{H(z^k,{\ov}^k)}(x^k)=x^k-\disp\frac{\big\la
T(z^k)+{\ov}^k,x^k-z^k\big\ra}{\|T(z^k)+{\ov}^k\|^2}
\big{(}T(z^k)+{\ov}^k\big{)}$. Then \eqref{fejer-des} yields
\begin{align}\label{ineq*}
\|x^{k+1}-x_{*}\|^2 \leq& \, \|x^k-x_{*}\|^2
-\left\|x^k-\frac{\big\la
T(z^k)+{\ov}^k,x^k-z^k\big\ra}{\|T(z^k)+{\ov}^k\|^2}
\big(T(z^k)+{\ov}^k\big)-x^k\right\|^2 \nonumber\\
=& \,\|x^k-x_{*}\|^2-\frac{(\la
T(z^k)+{\ov}^k,x^k-z^k\ra)^2}{\|T(z^k)+{\ov}^k\|^2}.\nonumber
\end{align}
It follows that
\begin{equation}\label{cero*}
\frac{\la
T(z^k)+{\ov}^k,x^k-z^k\ra^2}{\|T(z^k)+{\ov}^k\|^2} \leq
\|x^{k}-x_{*}\|^2- \|x^{k+1}-x_{*}\|^2.
\end{equation}
 Since  $T$ and the projection are continuous  and $(x^k)_{k\in \NN}$ is bounded, $(z^k)_{k\in \NN}$ is bounded.
 The boundedness of $\big(\|T(z^k)+{\ov}^k\|\big)_{k\in \NN}$ follows from \eqref{zk212}. Using Fact~\ref{punto}(ii),
 the right hand side of (\ref{cero*}) goes to $0$, when $k$ goes to $\infty$. Then, the result follows.
\endproof

Next we establish the main convergence result for {\bf Variant B.1}.
\begin{theorem}\label{teo1}
The sequence $(x^k)_{k\in \NN}$ converges to a point in $S_*$.
\end{theorem}
\proof
By Fact~\ref{punto}\ref{punto-iii}, we show that there exists an accumulation point of $(x^k)_{k\in \NN}$ belonging to $S_*$.
First, $(x^k)_{k\in \NN}$ is bounded due to Proposition~\ref{prop2}\ref{prop2-ii}. Let $(x^{i_k})_{k\in \NN}$ be a convergent subsequence such that $(u^{i_k})_{k\in \NN}$, $(v^{i_k})_{k\in \NN}$,
and $(\alpha_{i_k})_{k\in \NN}$ also converge. Set $\disp\lim _{k\to \infty}x^{i_k}= \tilde{x}$, $\disp\lim _{k\to \infty}u^{i_k}= \tilde{u}$,
$\disp\lim_{k\to \infty}v^{i_k}= \tilde{v}$ and $\disp\lim_{k\to \infty}\alpha^{i_k}= \tilde{\alpha}$. Using Proposition
\ref{prop2}\ref{prop2-iii}, \eqref{d:useful1}, and taking the limit as $k\to\infty$, we have
$$0=\disp\lim_{k\to \infty} \langle T(z^{i_k})+{\ov}^{i_k},x^{i_k}-z^{i_k}
\rangle \geq \frac{(1- \delta)}{\tilde{\alpha}}\lim_{k \to
\infty}\|z^{i_k}-x^{i_k}\|^2\geq 0.$$
This implies
\begin{equation}\label{limcero*}
\lim_{k\to \infty}\|x^{i_k}-z^{i_k}\|=0.
\end{equation}
Now we consider two cases:

\noindent {\bf Case~1:} $\disp\lim_{k\to \infty}\alpha_{i_k}=\tilde{\alpha}>0$. From \eqref{zk212}, the continuity of $T$ and the projection, and \eqref{limcero*}, we have
 $\tilde{x}=\disp\lim_{k\to \infty}
x^{i_k}=\lim_{k\to \infty}
z^{i_k}=P_C\big(\tilde{x}-\tilde{\alpha}(T(\tilde{x})+\tilde{\alpha}\tilde{u})\big)$. So $\tilde{x}\in S_*$ due to Proposition \ref{parada}.

\noindent {\bf Case~2:} $\disp\lim_{k \to \infty}\alpha_{i_k}=\tilde{\alpha}=0$. Define $\tilde{\alpha}_{k}:=\frac{\alpha_{k}}{\theta}$, then $\disp\lim_{k\to\infty}\tilde{\alpha}_{i_k}=0$. So we can assume $\tilde{\alpha}_{i_k}$ does not satisfy Armijo-type
condition in {\bf Linesearch~\ref{boundary}}, i.e.,
\begin{equation}\label{nosatisf}
\|T\big( \tilde{z}^{i_k}\big)-T(x^{i_k})+ \tilde{\alpha}_{i_k} \tilde{v}^{i_k}-\tilde{\alpha}_{i_k} u^{i_k}\|
>\frac{ \delta \| \tilde{z}^{i_k}-x^{i_k}\|}{\tilde{\alpha}_{i_k}},
\end{equation}
where $\tilde{v}^{i_k}\in \mcN_C(\tilde{z}^{i_k})$ and
$\tilde{z}^{i_k}=P_C(x^{i_k}-\tilde{\alpha}_{i_k}(T(x^{i_k})+\tilde{\alpha}_{i_k}u^{i_k}))$. The left hand side of \eqref{nosatisf} goes to $0$ by the continuity of $T$ and $P_C$. So,
\begin{equation}\label{frac-to-0}
\lim_{k\to \infty}\frac{\|
\tilde{z}^{i_k}-x^{i_k}\|}{\tilde{\alpha}_{i_k}}=0.
\end{equation}
By Corollary \ref{projP1}, with $x=x^{i_k}$, $\alpha=\tilde{\alpha}_{i_k}$
and $p=T(x^{i_k})+\tilde{\alpha}_{i_k} u^{i_k}$, we have
\begin{equation*}
\disp \frac{x^{i_k}-\tilde{z}^{i_k}}{\tilde{\alpha}_{i_k}}\in
T(x^{i_k})+\tilde{\alpha}_{i_k}u^{i_k}+\mcN_C(\tilde{z}^{i_k}).
\end{equation*}
Taking the limits as $k\to\infty$ and using \eqref{frac-to-0},
the continuity of $T$ and the closedness of ${\rm Gph}(\mcN_C)$, we obtain $0\in T(\tilde{x})+\mcN_C(\tilde{x})$, thus, $\tilde{x} \in S_*$.
\endproof

\subsection{Convergence Analysis of Variant
B.2}\label{sec-5.2}

We consider the case {\bf Variant~B.2} is used and the algorithm generates an infinite sequence $(x^k)_\kkk$ such that $x^k\notin S_*$ for all $\kkk$.
\begin{proposition}\label{fe*}
The sequence $(x^k)_{k\in \NN}$ is F\'ejer convergent to $S_*$.
Moreover, it is bounded and $\disp\lim_{k \to \infty} \|x^{k+1}-x^k\|=0$.
\end{proposition}
\proof
Take $x_{*}\in S_*$. By Lemma \ref{propseq}, $x_{*}\in
H(z^k,{\ov}^k)$, for all $k\in\NN$. Moreover $x_*\in C$ implies that the projection step \eqref{P122} is well-defined.
Next, using Fact~\ref{proj}\ref{proj-i} for two points $x^k$, $x_*$ and the set $C\cap H(z^k,\ov^k)$, we have
\begin{equation}\label{fejerc*}
\|x^{k+1}-x_{*}\|^2\le \|x^k-x_{*}\|^2-\|x^{k+1}-x^k\|^2.
\end{equation}
So, $(x^k)_{k\in \NN}$ is F\'ejer convergent to $S_*$. Hence,
$(x^k)_{k\in \NN}$ is bounded by Fact~\ref{punto}\ref{punto-i}. Taking the limit in (\ref{fejerc*}) and using Fact~\ref{punto}\ref{punto-ii}, we obtain the conclusion.
\endproof

\medskip
The next proposition shows a connection between the projection steps in {\bf Variant B.1} and {\bf Variant B.2}.
This fact has a geometry interpretation: in {\bf Variant B.2}, $x^k$ is projected onto a smaller set, thus, it may improve the convergence.

\begin{proposition}\label{p:0823a}
The following hold
\begin{enumerate}
\item\label{p:0823a-i} $x^{k+1}=P_{C\cap H(z^k,{\ov}^k)}(P_{H(z^k,{\ov}^k)}(x^k))$.
\item\label{p:0823a-ii} $\disp\lim_{k\to \infty}\langle T(z^k)+{\ov}^k,x^k-z^k \rangle =0$.
\end{enumerate}
\end{proposition}
\proof
\ref{p:0823a-i}:  Since $x^k \in C$ but $x^k\notin H(z^k,{\ov}^k)$ and $C\cap H_{k}\neq\emp$, the result follows from Lemma \ref{setprop}.

\noindent \ref{p:0823a-ii}: Take $x_{*}\in  S_*$. Notice that $x^{k+1}=P_{C\cap H(z^k,{\ov}^k)}(x^k)$ and that
projections onto convex sets are firmly-nonexpansive (see Fact~\ref{proj}\ref{proj-i}), we have
\begin{equation*}
\|x^{k+1}-x_{*}\|^2
=\|x^{k}-x_*\|^2-\|x^{k+1}-x^k\|^2
\leq\|x^k-x_{*}\|^2-\|P_{H(z^k,{\ov}^k)}(x^k)-x^k\|^2.
\end{equation*}
The remainder of the proof is analogous to Proposition \ref{prop2}\ref{prop2-iii}.
\endproof

Finally we present the convergence result for {\bf Variant B.2}.
\begin{proposition}\label{teo2*}
The sequence $(x^k)_{k\in \NN}$ converges to a point in $S_*$.
\end{proposition}
\proof
Similar to the proof of Theorem \ref{teo1}.
\endproof

\subsection{Convergence Analysis of Variant
B.3}\label{sec-5.3}

We consider the case {\bf Variant
B.3} is used and the algorithm generates an infinite sequence $(x^k)_\kkk$ such that $x^k\notin S_*$ for all $\kkk$. Observe that $C\cap H(z^k,{\ov}^k)\cap W(x^k)$ is a closed convex set. So, the algorithm is well-defined if
this set $C\cap H(z^k,{\ov}^k)\cap W(x^k)$. The following lemma guarantees its non-emptiness.

\begin{lemma}\label{lemma:3}
For all $\kkk$, we have \
$S_*\subseteq C\cap H(z^k,{\ov}^k)\cap W(x^k)$.
\end{lemma}
\proof We proceed by induction. By definition, $\emp\neq S_*\subseteq C$.
By Lemma~\ref{propseq}, $S_*\subseteq  H(z^k,{\ov}^k)$ for all $k$. Since $W(x^0)=\RR^n$, we have $S_*\subseteq
H(z^0,{\ov}^0)\cap W(x^0)$.
Assume that $S_*\subseteq H(z^k,{\ov}^k) \cap W(x^k)$. Then, $x^{k+1}=P_{C\cap H(z^k,{\ov}^k)\cap W(x^k)}(x^0)$ is well-defined.
By Fact~\ref{proj}\ref{proj-ii}, we obtain $\langle x_{*}-x^{k+1}\,,\, x^0-x^{k+1}\rangle\leq 0$ for all $x_{*}\in S_*$.
This implies $x_{*}\in W(x^{k+1})$. Hence, $S_* \subseteq H(z^{k+1},{\ov}^{k+1})\cap W(x^{k+1})$. Then,
the conclusion follows by induction principle.
\endproof

Before proving the convergence of the sequence $(x^k)_\kkk$, we
study its boundedness. The next lemma  shows that the sequence
remains in a ball determined by the initial point.
\begin{lemma}\label{l:lim}
Let $\ox=P_{S_*}(x^0)$ and $\rho={\rm dist}(x^0, S_*)$. Then $
(x^k)_{k\in \NN}\subset
B\left[\frac{1}{2}(x^0+\ox),\frac{1}{2}\rho\right]\cap C$,
in particular, $(x^k)_\kkk$ is bounded.
\end{lemma}
\proof
By Lemma~\ref{lemma:3}, we have $S_* \subseteq H(z^k,{\ov}^k) \cap W(x^k)$ for all $k$. Using Lemma~\ref{l:lim-2}, with $S=S_*$ and $x=x^k$, we obtain $x^k\in B\left[\frac{1}{2}(x^0+\ox),\frac{1}{2}\rho\right]$ for all $\kkk$. Finally, notice that $(x^k)_{k\in \NN}\subset C$.
\endproof

Now, we focus on the properties of the accumulation points.
\begin{proposition}\label{p:opt} All accumulation points of $(x^k)_{k\in \NN}$ belong to $S_*$.
\end{proposition}
\proof
Since $W(x^k)$ is a halfspace with normal $x^0-x^k$, we have $x^k=P_{W(x^k)}(x^0)$. So by the firm non-expansiveness of $P_{W(x^k)}$ (see Fact~\ref{proj}\ref{proj-i}) and $x^{k+1}\in W(x^k)$, we have $$\|x^{k+1}-x^k\|^2\leq\|x^{k+1}-x^0\|^2-\|x^k-x^0\|^2.$$
Thus, $(\|x^k-x^0\|)_{k\in \NN}$ is monotone and nondecreasing. Moreover, by Lemma \ref{l:lim}, $(\|x^k-x^0\|)_{k\in \NN}$ is bounded, thus, converges. It follows that
\begin{equation}\label{eqb}
\lim_{k\rightarrow\infty}\| x^{k+1}-x^k\|=0.
\end{equation}
Since $x^{k+1}\in H(z^k,{\ov}^k)$, we get $\langle T(z^{k})+{\ov}^{k},x^{k+1}-z^{k}\rangle \le 0$,
where $z^k$ and ${\ov}^k$ are obtained in Steps~2 and~3, respectively. Combining with \eqref{d:useful1}, we obtain
\begin{equation*}
\begin{aligned}
0&\geq\langle T(z^{k})+{\ov}^{k},x^{k+1}-x^{k}\rangle + \big\la
T(z^{k})+{\ov}^{k},x^{k}-z^{k}\big\rangle \\
&\geq
-\|T(z^{k})+{\ov}^{k}\|\cdot\|x^{k+1}-x^{k}\|+
\frac{1-\delta}{\alpha_k}
\|x^k-z^k\|^2.
\end{aligned}
\end{equation*}
Using \eqref{zk212} and some simple algebra,
\begin{equation}
\label{eq}
\|x^k-z^k\|^2\le\frac{\sigma}{1-\delta}\|T(z^{k})+{\ov}^{k}\|\cdot\|x^{k+1}-x^{k}\|.
\end{equation}
By the boundedness of
$({\ov}^k)_{k\in \NN}$ and $(x^{k})_{k\in \NN}$, we can
choose a subsequence $(i_k)_{k\in \NN}$ such that $(\alpha_{i_k})_{k\in \NN}$, $(x^{i_k})_{k\in \NN}$, and $({\ov}^{i_k})_{k\in \NN}$
converge to $\tilde{\alpha}$, $\tilde{x}$, and $\tilde{v}$, respectively. Taking the limits in \eqref{eq} and using \eqref{eqb}, we get
$\disp\lim_{k\to \infty}\|x^{i_k} - z^{i_k}\|^2=0$. Consequently, $\disp\tilde{x}=\lim_{k\to\infty}z^{i_k}$.
Now we consider two cases:

\noindent {\bf Case~1:} $\disp\lim_{k\to\infty} \alpha_{i_k} =\tilde{\alpha}> 0$. By \eqref{zk212} and the continuity of the projection, $\disp\tilde{x}=\lim_{k\to
\infty}z^{i_k}=P_C\big(\tilde{x}-\tilde{\alpha}(T(\tilde{x})+\tilde{\alpha}\tilde{u})\big)$
and hence by Proposition \ref{parada}, $\tilde{x}\in S_*$.

\noindent {\bf Case~2:} $\disp\lim_{k\to\infty} \alpha_{i_k} = \tilde{\alpha}=0$. This case is similar to the proof of Theorem~\ref{teo1}.
\endproof

\medskip
Finally, we prove that $(x^k)_{k\in \NN}$ converges to the solution closest to $x^0$.

\begin{theorem}\label{Teo-F} The sequence $(x^k)_{k\in \NN}$ converges to $\ox=P_{S_*}(x^0)$.
\end{theorem}
\proof  First, $\bar{x}$ is well-defined due to Lemma \ref{sol-convex-closed}. It follows from Lemma \ref{l:lim} that $(x^k)_{k\in \NN}\subset
B\left[\frac{1}{2}(x^0+\ox),\frac{1}{2}\rho\right]\cap C$  where $\rho={\rm dist}(x^0, S_*)$, so it
is bounded. Let $(x^{i_k})_{k\in \NN}$ be a subsequence
of $(x^k)_{k\in \NN}$ that converges to $\hat{x}$. Then,
$\hat{x}\in
B\left[\frac{1}{2}(x^0+\ox),\frac{1}{2}\rho\right]\cap C$.
Furthermore, $\hat{x}\in S_*$ due to Proposition~\ref{p:opt}. So,
$\hat{x}\in S_*  \cap  B
\left[\frac{1}{2}(x^0+\ox),\frac{1}{2}\rho\right]=\{\ox\}$. Thus, $\ox$ is the unique accumulation point of $(x^k)_{k\in \NN}$. Hence, $(x^{k})_{k\in \NN}$ converges
to $\ox\in S_*$. \endproof

\section{Conceptual Algorithm with Linesearch \ref{feasible}}\label{sec-5}

As mentioned before, the disadvantage of {\bf Linesearch~\ref{boundary}} is the necessity to compute the projection onto the feasible set within the
inner loop to find the stepsize $\alpha$. To overcome this, we propose the second conceptual algorithm that uses a linesearch along feasible directions. 

We further note that in {\bf Linesearch~\ref{feasible}} below, if we set $u=0\in\mcN_C(x)$, then the projection step is done {\em outside} the {\bf While} loop.
\vspace*{-.15in}
\begin{center}\fbox{\begin{minipage}[b]{\textwidth}
\begin{linesr}{F}
{\rm(Linesearch along the feasible direction)}
\label{feasible}

\medskip
{\bf Input:} $(x,u,\beta,\delta,M)$. Where $x\in C$, $u\in \mcN_C(x)$, $\beta>0$, $\dd\in(0,1)$, and $M>0$.

Set $\alpha\leftarrow 1$ and $\theta \in (0,1)$. Define $z_{\alpha}=P_C(x-\beta(T(x)+\alpha u))$ and choose
$u\in \mcN_C(x)$, $v_1\in\mcN_C(z_1)$ with $\|v_1\|\leq M$.
\begin{retraitsimple}
\item[] {\bf While} $\langle T
\big(\alpha z_{\alpha}+(1-\alpha)x\big)+v_{\alpha}, x-z_{\alpha}\ra
< \delta \langle T
(x)+ \alpha u, x-z_{\alpha} \ra$ {\bf do}

$\alpha\leftarrow\theta \alpha$ \ and \ choose any $v_{\alpha}\in\mcN_C(\alpha z_{\alpha}+(1-\alpha)x)$ with $\|v_\alpha\|\leq M$.

\item[] {\bf End While}
\end{retraitsimple}
{\bf Output:} $(\alpha,z_\alpha,v_{\alpha})$.
\end{linesr}\end{minipage}}\end{center}
Again, {\bf Linesearch \ref{feasible}} is
also well-defined assuming only \ref{a1}, i.e., continuity of $T$.
\begin{lemma}\label{feasible-well}
If $x\in C$ and $x\notin S_*$, then {\bf Linesearch \ref{feasible}} stops after finitely many steps.
\end{lemma}
\proof
Suppose on the contrary that {\bf Linesearch \ref{feasible}} does not stop for all $\alpha\in \mcP:=\{1, \theta, \theta^2, \ldots\}$ and that
\begin{equation}
v_{\alpha}\in\mcN_C\big(\alpha z_\alpha+(1-\alpha)x\big),\quad
\|v_\alpha\|\leq M,\quad
z_\alpha=P_C\big(x-\beta(T(x)+ \alpha u)\big).\label{vz_alpha}
\end{equation}
We have
\begin{equation}\label{no-arm}
\langle T (\alpha z_\alpha+(1-\alpha)x)+v_\alpha, x-z_\alpha\rangle < \delta
\langle T (x)+\alpha u, x-z_\alpha \ra.
\end{equation}
By \eqref{vz_alpha}, the sequence $(v_\alpha)_{\alpha\in \mcP}$ is bounded. Thus, without loss of generality, we can assume that it converges to some
$v_0\in\mcN_C(x)$ (by Fact~\ref{nor-cone}). The continuity of the projection operator and the formula of $z_\alpha$ in
\eqref{vz_alpha} imply that $(z_\alpha)_{\alpha\in \mcP}$ converges to $z_0=P_C(x-\beta T(x))$. Taking the
limit in \eqref{no-arm} as $\alpha\to 0$, we get $\langle T(x)+ v_0, x -z_0 \rangle \le \delta \langle T(x), x -z_0
\ra$. It follows that
\begin{equation*}
0\geq (1-\delta)\langle  T(x), x -z_0 \ra+\langle v_0, x -z_0 \rangle \ge
(1-\delta)\langle  T(x), x -z_0 \rangle \ge \frac{(1-\delta)}{\beta}
\|x-z_0\|^2.
\end{equation*}
So, $x=z_0=P_C(x-\beta T(x))$, i.e., $x\in S_*$, a contradiction.
\endproof

Next, we present the conceptual algorithm, which is related to {\bf Algorithm~\ref{Extragradient}} with Strategy ({\bf c}) when nonzero normal vectors are used. Here, we assume that \ref{a1} and \ref{a2} hold.
\begin{center}\fbox{\begin{minipage}[b]{\textwidth}
\begin{Calg}{F}\label{A1}
Given $(\beta_k)_{k\in \NN}\subset[\check{\beta},\hat{\beta}]$ $0<\check{\beta}\le \hat{\beta}<+\infty$, $\dd\in(0,1)$, and $M>0$.
\item[ ]{\bf Step~0 (Initialization):} Take $x^0\in C$ and set $k\leftarrow 0$.

\item[ ]{\bf Step~1 (Stopping Test):} If $x^k=P_C(x^k-T(x^k))$, then stop. Otherwise,

\item[ ]{\bf Step 2 (Linesearch \ref{feasible}):} Take $u^k\in \mcN_C(x^k)$ with $\|u^k\|\leq M$ and
set
\begin{equation}\label{vbar}
(\alpha_k,z^k,{\ov}^{k})= {\bf Linesearch\; \ref{feasible}}\;(x^k, u^k,\beta_k,\dd,M),
\end{equation}
i.e., $(\alpha_k,z^k,\bar{v}^k)$ satisfies
\begin{subequations}\label{zk212*}
\begin{align}
&\bar{v}^k\in \mcN_C(\alpha_k z^k+(1-\alpha_k)x^k) \ \mbox{with} \ \|\bar{v}^k\|\leq M,\quad \alpha_k\leq1,\label{zk212*-a}\\
& z^k=P_C(x^{k}-\beta_k(T(x^{k})+\alpha_k u^k)),\label{zk212*-b}\\
&\langle T(\alpha_k z^k+(1-\alpha_k)x^k)+\bar{v}^k,x^k-z^k\rangle \geq\delta\langle T(x^k)+\alpha_ku^k,x^k-z^k\ra.\label{zk212*-c}
\end{align}
\end{subequations}
\item[ ]{\bf Step 3 (Projection):} Set \
$\ox^k:=\alpha_k z^k+(1-\alpha_k)x^k$ \ and \ $x^{k+1}:=\mcF_F(x^k)$.

\item[ ]{\bf Step 4:} Set $k\leftarrow k+1$ and go to {\bf Step~1}.
\end{Calg}\end{minipage}}\end{center}
We also consider three variants of $\mcF_F$ in Step~3:
\begin{align}
\mcF_{\rm F.1}(x^k) =& P_C\big(P_{H(\ox^k,{\ov}^k)}(x^k)\big),\label{P11} \quad   &{(\bf Variant\; F.1)} \\
\mcF_{\rm F.2}(x^k) =& P_{C\cap H(\ox^k,{\ov}^k)}(x^k),\label{P12}\quad  &{(\bf Variant\; F.2)}\\
\mcF_{\rm F.3}(x^k) =& P_{C\cap H(\ox^k,{\ov}^k)\cap
W(x^k)}(x^0),\quad  & {(\bf Variant\; F.3)}\label{P13}
\end{align}
where, similar to \eqref{e:HW},
\begin{subequations}\label{e:HWF}
\begin{align}
H(\ox^k,\ov^k)&:=\big\{ y\in \RR^n : \langle T(\ox^k)+\ov^k,y-\ox^k\rangle \le 0\big \},\label{H(x,v)F}\\
\text{and}\qquad
W(x^k)&:=\big\{ y\in \RR^n : \langle y-x^k,x^0-x^k\rangle \le 0\big \}.\label{W(x)F}
\end{align}
\end{subequations}
Now, we analyze some general properties of {\bf Conceptual Algorithm~\ref{A1}}.
\begin{proposition}\label{propdef}
Assuming that $\mcF_F(x^k)$ is well-defined whenever $x^k$ is available. Then, {\bf Conceptual Algorithm \ref{A1}} is well-defined.
\end{proposition}
\proof
If Step~1 is not satisfied, then Step~2 is guaranteed by Lemma~\ref{feasible-well}. Thus, the entire algorithm is well-defined.
\endproof

\begin{proposition}\label{H-separa-x12} 
$x^k\in S_*$ if and only if
$x^k \in H(\ox^k,{\ov}^k)$ where ${\ov}^k$ and $\ox^k$ are given in Steps~2 and~3, respectively,
\end{proposition}
\proof
If $x^k \in S_*$, then $x^k \in H(\ox^k,{\ov}^k)$ by Lemma~\ref{propseq}. Conversely, suppose $x^k\in H(\ox^k,{\ov}^k)$, $\langle T(\ox^k)+{\ov}^k, x^k-\ox^k\rangle \le 0$. Using the definitions of $\alpha_k$ and $\ox^k$ in Steps~2 and~3, we have
\begin{equation*}
0\geq\, \langle T(\ox^k)+{\ov}^k, x^k-\ox^k\rangle =\alpha_k\la
T(\ox^k)+{\ov}^k,x^k-z^k\rangle \ge\alpha_k\delta\big\langle T(x^k)+\alpha_k u^k, x^k -z^k\big \ra.
\end{equation*}
From the definition of $z^k$ in Step~2, we derive
\begin{equation*}
\alpha_k\delta\big\langle T(x^k)+\alpha_k u^k, x^k -z^k\big \rangle \geq
\frac{\alpha_k\delta}{\hat{\beta_k}}\|x^k-z^k\|^2
\ge \,\frac{\alpha_k\delta}{\hat{\beta}}\|x^k-z^k\|^2.
\end{equation*}
It follows that $0\geq\|x^k-z^k\|^2$, i.e., $x^k=z^k$. Now from Proposition~\ref{parada}, we conclude $x^k\in S_*$.
\endproof

\medskip
Let $(x^k)_{k\in \NN}$ and $(\alpha_k)_{k\in \NN}$
be sequences generated by {\bf Conceptual Algorithm~\ref{A1}} and suppose that $x^k\notin S_*$. From the proof of Proposition~\ref{H-separa-x12}, we obtaina useful algebraic property
\begin{equation}\label{d:useful}
\forall k\in\NN:\quad\langle T(\ox^{k})+{\ov}^{k},x^{k}-\ox^{k} \rangle  \geq
\frac{\alpha_k \delta}{\hat{\beta}}\|x^k-z^k\|^2.
\end{equation}

\begin{proposition}\label{p:aFeas_xk1_neq_xk}
If Stopping Test is not satisfied at $x^k$, then {\bf Conceptual Algorithm~\ref{feasible}} generates $x^{k+1}\neq x^k$.
\end{proposition}
\proof Suppose on the contrary that $x^{k+1}=x^k$. Consider three cases.

If {\bf Variant~F.1} is used, then $x^{k+1}=P_C\big(P_{H(\ox^k,{\ov}^k)}(x^k)\big)=x^k$. So Fact~\ref{proj}\ref{proj-ii} implies
\begin{equation}\label{proyex}
\forall z\in C:\quad\langle P_{H(\ox^k,{\ov}^k)}(x^k)-x^k, z-x^k\rangle \leq 0.
\end{equation}
Again, using Fact~\ref{proj}\ref{proj-ii},
\begin{equation}\label{proyeh}
\forall z\in H(\ox^k,{\ov}^k):\quad
\langle P_{H(\ox^k,{\ov}^k)}(x^k)-x^k,
P_{H(\ox^k,{\ov}^k)}(x^k)-z\rangle \leq 0.
\end{equation}
Note that $\emp\neq S_*\subseteq C\cap H(\ox^k,{\ov}^k)$ by Proposition~\ref{H-separa-x12}. So, taking any $z\in C\cap H(\ox^k,{\ov}^k)$, then adding up \eqref{proyex} and \eqref{proyeh}, we derive $\|x^k-P_{H(\ox^k,{\ov}^k)}(x^k)\|^2=0$.
Hence, $x^k=P_{H(\ox^k,{\ov}^k)}(x^k)$, i.e., $x^k\in H(\ox^k,{\ov}^k)$. 

If {\bf Variant~F.2} is used, then $x^{k+1}=P_{C\cap H(\ox^k,{\ov}^k)}(x^k)=x^k$. So $x^k\in H(\ox^k,{\ov}^k)$.

If {\bf Variant~F.3} is used, then $x^{k+1}=P_{C\cap H(\ox^k,{\ov}^k)\cap W(x^k)}(x^0)=x^k$. So $x^k\in H(\ox^k,{\ov}^k)$.

Hence, in all cases, we have showed that $x^k\in H(\ox^k,\ov^k)$, which means $x^k\in S_*$ by Proposition~\ref{H-separa-x12}. By Fact~\ref{f:0822b}, we get $x^k=P_C(x^k-T(x^k))$, i.e., Stopping Test is satisfied at $x^k$, a contradiction.
\endproof

In view of Proposition~\ref{p:aFeas_xk1_neq_xk}, we will again examine only the case that Stopping Test is not satisfied for all $x^k$. In this case, {\bf Conceptual Algorithm~\ref{A1}} generates an infinite sequence $(x^k)_{k\in \NN}$ such that $x^k\notin S_*$ for all $\kkk$.

\subsection{Convergence Analysis of Variant~F.1}\label{sec-4.1}

We consider the case {\bf Variant~F.1} is used and the algorithm generates an infinite sequence $(x^k)_{k\in \NN}$ such that
$x^k\not\in S_*$ for all $k\in\NN$. Note that by Lemma~\ref{propseq}, $H(\ox^k,{\ov}^k)$ is nonempty for all $k$.
Then, the projection step~\eqref{P11} is well-defined and so is the entire algorithm.

\begin{proposition}\label{prop222}
The following hold:
\begin{enumerate}
\item\label{prop222-i}  The sequence $(x^k)_{k\in \NN}$ is Fej\'er convergent to $S_*$.
\item\label{prop222-ii} The sequence $(x^k)_{k\in \NN}$ is bounded.
\item\label{prop222-iii} $\disp\lim_{k\to \infty}\langle T(\ox^k)+{\ov}^k,x^k-\ox^k \rangle =0$.
\end{enumerate}
\end{proposition}
\proof
\ref{prop222-i}: Take $x_{*}\in S_*$. Note that, by definition $(\ox^k,{\ov}^k) \in {\rm Gph}(\mcN_C)$.
Using \eqref{P11}, Fact~\ref{proj}(i) and Lemma \ref{propseq}, we have
\begin{align}\label{fejer-des1}\nonumber
\|x^{k+1}-x_{*}\|^2&=\|P_{C}(P_{H(\ox^k,{\ov}^k)}(x^k))-
P_{C}(P_{H(\ox^k,{\ov}^k)}(x_{*}))\|^2\\
&\leq\|P_{H(\ox^k,{\ov}^k)}(x^k)- P_{H(\ox^k,{\ov}^k)}(x_{*})\|^2\\
&\leq \|x^k-x_{*}\|^2-\|P_{H(\ox^k,{\ov}^k)}(x^k)-x^k\|^2
\leq \|x^k-x_{*}\|^2.
\end{align}

\ref{prop222-ii}: Follows immediately from \eqref{prop222-i} and Fact~\ref{punto}\ref{punto-i}.

\ref{prop222-iii}: Take $x_{*} \in S_*$. Using $\disp
P_{H(\ox^k,{\ov}^k)}(x^k)=x^k-\frac{\big\langle
T(\ox^k)+{\ov}^k,x^k-\ox^k\big\rangle}{\|T(\ox^k)+{\ov}^k\|^2}
\big{(}T(\ox^k)+{\ov}^k\big{)}$, \eqref{fejer-des1}, and the definition of $\ox^k$ in Step~3, we derive
\begin{align*}
\|x^{k+1}-x_{*}\|^2 &\leq \|x^k-x_{*}\|^2 -\left\|x^k-
\frac{\big\langle T(\ox^k)+{\ov}^k,x^k-\ox^k\big\rangle}{\|T(\ox^k)+{\ov}^k\|^2} \big(T(\ox^k)+{\ov}^k\big)-x^k\right\|^2\\
&=\|x^k-x_{*}\|^2-\frac{\la
T(\ox^k)+{\ov}^k,x^k-\ox^k\ra^2}{\|T(\ox^k)+{\ov}^k\|^2}.
\end{align*}
It follows that
$\disp\frac{\la
T(\ox^k)+{\ov}^k,x^k-\ox^k\ra^2}{\|T(\ox^k)+{\ov}^k\|^2}\le
\|x^{k}-x_{*}\|^2- \|x^{k+1}-x_{*}\|^2\to 0$.
By Fact~\ref{punto}\ref{punto-ii}, the right hand side goes to zero as $k\to\infty$. Since $T$ is continuous and $(x^k)_{k\in \NN}$, $(z^k)_{k\in \NN}$
and $(\ox^k)_{k\in \NN}$ are bounded, $\big(\|T(\ox^k)+{\ov}^k\|\big)_{k\in \NN}$ is also bounded. So the conclusion follows.
\endproof

Next, we establish our main convergence result for {\bf Variant
F.1}.
\begin{theorem}\label{teo12}
The sequence $(x^k)_{k\in \NN}$ converges to a point in $S_* $.
\end{theorem}
\proof
By Fact~\ref{punto}\ref{punto-iii}, we show that there exists an accumulation point of $(x^k)_{k\in \NN}$ belonging to $S_*$.
First, $(x^{k})_{\kkk}$ is bounded due to Proposition~\ref{prop222}\ref{prop222-ii}. Let $(x^{i_k})_{k\in
\NN}$ be a convergent subsequence of $(x^k)_{k\in \NN}$ such that, $(\ox^{i_k}), ({\ov}^{i_k}),(u^{i_k}), (\alpha_{i_k})_{k\in
\NN}$, and $(\beta_{i_k})_{k\in \NN}$ also converge. Set $\disp\lim
_{k\to \infty}x^{i_k}= \tilde{x}$, $\disp\lim _{k\to \infty}u^{i_k}=
\tilde{u}$, $\disp\lim _{k\to \infty}\alpha_{i_k}= \tilde{\alpha}$, and
$\disp\lim_{k\to \infty}\beta_{i_k}= \tilde{\beta}$. Using Proposition~\ref{prop222}\ref{prop2-iii}, \eqref{d:useful},
and taking the limit as $k\to\infty$, we derive
$\disp
0=\lim_{k\to \infty}\langle T(\ox^{i_k})+\bar{u}^{i_k},x^{i_k}-\ox^{i_k} \rangle \ge \lim_{k\to \infty}
\frac{\alpha_{i_k}}{\hat{\beta}}\delta\|x^{i_k}-z^{i_k}\|^2\geq0.
$
Therefore,
\begin{equation}\label{lim-0}
\lim_{k\to \infty} \alpha_{i_k}\|x^{i_k}-z^{i_k}\|=0.
\end{equation}
Now we consider two cases.

\noindent {\bf Case~1:} $\disp\lim_{k\to \infty}\alpha_{i_k}=\tilde{\alpha}>0$. From \eqref{lim-0},
the continuity of $T$ and the projection, we obtain $\disp\tilde{x}=\lim_{k\to
\infty}
x^{i_k}=\lim_{k\to
\infty}
z^{i_k}=P_C\big(\tilde{x}-\tilde{\beta}(T(\tilde{x})+\tilde{\alpha}\tilde{u})\big)$.
So, $\tilde{x}\in S_*$ by Proposition \ref{parada}.

\noindent {\bf Case~2:} $\disp\lim_{k\to \infty}\alpha_{i_k}=\tilde{\alpha}=0$. Define
$\tilde{\alpha}_{i_k}=\frac{\alpha_{i_k}}{\theta}$. Then, $\disp\lim_{k\to\infty}\tilde{\alpha}_{i_k}=0$. So we can assume $\tilde{\alpha}_{i_k}$ does not satisfy Armijo-type condition in {\bf Linesearch~\ref{feasible}}, i.e.,
\begin{equation}\label{conse}
\Big\langle T(\tilde{y}^{i_k}) + \tilde{v}^{i_k}, x^{i_k} -
\tilde{z}^{i_k}\Big\ra<\delta \langle T (x^{i_k})+\tilde{\alpha}_{i_k}u^{i_k},
x^{i_k}-\tilde{z}^{i_k}\rangle ,
\end{equation}
where
$\tilde{y}^{i_k}:=\tilde{\alpha}_{i_k}\tilde{z}^{i_k}+(1-\tilde{\alpha}_{i_k})x^{i_k}$,
$\tilde{z}^{i_k}=P_C(x^{i_k}-\beta_{i_k}(T(x^{i_k})+\tilde{\alpha}_{i_k} u^{i_k}))$, and $\tilde{v}^{i_k}\in\mcN_C(\tilde{y}^{i_k})$ with $\|\tilde{v}^{i_k}\|\leq M$. Hence,
$\tilde{y}^{i_k}\to\tilde{x}$. Next, taking a subsequence without relabeling, we assume that $\disp\lim_{k\to\infty}\tilde{v}^{i_k}=\tilde{v}$. So $\tilde{v}\in \mcN_C(\tilde{x})$ by Fact~\ref{nor-cone}.  Moreover, $\disp\lim_{k\rightarrow\infty}\tilde{z}^{i_k}=\tilde{z}=
P_C\big(\tilde{x}-\tilde{\beta}T(\tilde{x})\big)$ by the continuity of $T$ and $P_C$. Thus, passing to the limit in \eqref{conse}, we get $\langle T(\tilde{x}) + \tilde{v}, \tilde{x}-\tilde{z}\rangle \le
\delta \langle T(\tilde{x}) , \tilde{x}-\tilde{z}\ra$.
It follows that
\begin{align*}
0 &\geq 
\langle T(\tilde{x}) + \tilde{v}, \tilde{x}-\tilde{z}\rangle -
\delta \langle T(\tilde{x}) , \tilde{x}-\tilde{z}\rangle \\
&= (1-\delta) \big\langle T(\tilde{x}), \tilde{x}-\tilde{z}\big\rangle + \big\la\tilde{v}, \tilde{x}-\tilde{z}\big\ra
\ge (1-\delta) \big\langle T(\tilde{x}), \tilde{x}-\tilde{z}\big\rangle \\
&=\frac{(1-\delta)}{\tilde{\beta}} \big\la
\tilde{x}-(\tilde{x}-\tilde{\beta} T(\tilde{x})), \tilde{x}
-\tilde{z}\rangle \ge
\frac{(1-\delta)}{\tilde{\beta}}\|\tilde{x}-\tilde{z}\|^2
\ge  \,\frac{(1-\delta)}{\hat{\beta}}\|\tilde{x}-\tilde{z}\|^2 \ge
0.\nonumber
\end{align*}
This means $\tilde{x}=\tilde{z}$, which implies $\tilde{x}\in S_*$.
\endproof

\subsection{Convergence Analysis of Variant F.2}\label{sec-4.2}

We consider the case {\bf Variant~F.2} is used and the algorithm generates an infinite sequence $(x^k)_{k\in \NN}$ such that
$x^k\not\in S_*$ for all $k\in\NN$.

\begin{proposition}\label{fe}
The sequence $(x^k)_{k\in \NN}$ is F\'ejer convergent to $S_*$. Moreover, it is bounded and $\disp\lim_{k \to \infty} \|x^{k+1}-x^k\|=0$.
\end{proposition}
\proof
Take $x_{*}\in S_*\subseteq C$. By Lemma~\ref{propseq}, $x_{*}\in
H(\ox^k,{\ov}^k)$ for all $k$. So, the projection step \eqref{P12} is well-defined.
Then, using Fact~\ref{proj}\ref{proj-i} for the projection operator $P_{H(\ox^k,{\ov}^k)}$, we obtain
\begin{equation}\label{fejerc}
\|x^{k+1}-x_{*}\|^2\le \|x^k-x_{*}\|^2-\|x^{k+1}-x^k\|^2
\leq\|x^k-x_*\|^2.
\end{equation}
So $(x^k)_{k\in \NN}$ is F\'ejer convergent to $S_*$.
Thus, by Fact~\ref{punto}\ref{punto-i}\&\ref{punto-ii}, $(x^k)_{k\in \NN}$ is bounded and thus $(\|x^k-x_{*}\|)_{k\in \NN}$ is a convergent sequence. By passing to the limit in \eqref{fejerc} and using Fact~\ref{punto}\ref{punto-ii}, we get $\disp\lim_{k \to \infty} \|x^{k+1}-x^k\|=0$.
\endproof

\medskip
Again, in {\bf Variant~F.2}, $x^k$ is projected onto a smaller set than in {\bf Variant~F.1}, the former variant may improve the convergence.
\begin{proposition}\label{p:1027a}
Let $(x^{k})_{k\in \NN}$ be the sequence generated by {\bf Variant F.2}. Then,
\begin{enumerate}
\item\label{p:1027a-i} $x^{k+1}=P_{C\cap H(\ox^k,{\ov}^k)}(P_{H(\ox^k,{\ov}^k)}(x^k))$.
\item\label{p:1027a-ii} $\disp\lim_{k\to \infty}\langle T(\ox^k)+{\ov}^k,x^k-\ox^k \rangle =0$.
\end{enumerate}
\end{proposition}
\proof
(i):  Since $x^k \in C$ but $x^k\notin H(\ox^k,{\ov}^k)$  and $C\cap H_{k}\neq
\emp$, by Lemma \ref{setprop}, we have the result.

(ii): Take $x_{*}\in  S_*$. Notice that $x^{k+1}=P_{C\cap H(z^k,\ov^k)}(x^k)$ and that projections onto convex sets are
firmly-nonexpansive (see Fact~\ref{proj}\ref{proj-i}), we have
\begin{equation*}
\|x^{k+1}-x_{*}\|^2
\leq \|x^k-x_*\|^2-\|x^{k+1}-x^k\|^2
\leq\|x^k-x_{*}\|^2-\|P_{H(\ox^k,{\ov}^k)}(x^k)-x^k\|^2.
\end{equation*}
The rest of the proof is analogous to Proposition \ref{prop222}\ref{prop222-iii}.
\endproof

\begin{proposition}\label{teo2}
The sequence $(x^k)_{k\in \NN}$ converges to a point in $S_*$.
\end{proposition}
\proof
Similar to the proof of Theorem~\ref{teo12}.
\endproof

\subsection{Convergence Analysis of Variant F.3}\label{sec-4.3}

It is easy to check that $C\cap H(\ox^k,{\ov}^k)\cap W(x^k)$ is a closed convex set for each $k$. So, if $C\cap H(\ox^k,{\ov}^k)\cap W(x^k)$ is nonempty, then the next iterate $x^{k+1}$ is well-defined. The following lemma, whose proof is similar to Lemma \ref{lemma:3}, guarantees the non-emptiness.

\begin{lemma}\label{lemma:31}
For all $k\in\NN$, we have \
$S_*\subset C\cap  H(\ox^k,{\ov}^k)\cap W(x^k)$.
\end{lemma}
\proof We proceed by induction. By definition, $\emp\neq S_*\subseteq C$.
By Lemma~\ref{propseq}, $S_*\subseteq  H(\ox^k,{\ov}^k)$, for all $k$.
Since $W(x^0)=\RR^n$, we have $S_*\subseteq
H(\ox^0,{\ov}^0)\cap W(x^0)$.
Assume that $S_*\subseteq H(\ox^k ,{\ov}^k)\cap
W(x^k)$. So $x^{k+1}=P_{C\cap
H(\ox^k,{\ov}^k)\cap W(x^k)}(x^0)$ is well-defined. Then, by
Fact~\ref{proj}\ref{proj-ii}, we have $\langle x_{*}-x^{k+1}, x^0-x^{k+1}\rangle \leq0$ for all $x_{*}\in S_*$.
This implies $x_{*}\in W(x^{k+1})$, and hence, $S_* \subseteq H(\ox^{k+1},{\ov}^{k+1})\cap W(x^{k+1})$. Thus, the conclusion
follows by induction principle.
\endproof

The next lemma shows that the sequence $(x^k)_\kkk$ remains in a ball determined by the initial point.

\begin{lemma}\label{l:lim1} 
Let $\ox=P_{S_*}(x^0)$ and $\rho={\rm dist}(x^0, S_*)$.
Then $(x^k)_{k\in \NN}\subset
B\left[\frac{1}{2}(x^0+\ox),\frac{1}{2}\rho\right]\cap C$, in particular, $(x^k)_{k\in \NN}$ is bounded.
\end{lemma}
\proof
It follows from Lemma \ref{lemma:31} that $S_* \subseteq H(\ox^k,{\ov}^k) \cap W(x^k)$, for all $k\in\NN$. The remaining argument is similar to the proof of Lemma~\ref{l:lim}.
\endproof

\begin{theorem}\label{l:optimalidad2} All accumulation points of $(x^k)_{k\in\NN}$ belong to $S_*$. 
\end{theorem}
\proof
Since $W(x^k)$ is a halfspace with normal $x^0-x^k$, we have $x^k=P_{W(x^k)}(x^0)$.
So, by the firm nonexpansiveness of $P_{W(x^k)}$ and $x^{k+1}\in W(x^k)$, we have $\|x^{k+1}-x^k\|^2\leq\|x^{k+1}-x^0\|^2-\|x^k-x^0\|^2$. Thus, $(\|x^k-x^0\|)_{k\in \NN}$ is monotone and nondecreasing.
Moreover, by Lemma~\ref{l:lim1}, $(\|x^k-x^0\|)_\kkk$ is bounded, thus, converges. It follows that
\begin{equation}\label{dif:to:0}
\lim_{k\rightarrow\infty}\| x^{k+1}-x^k\|=0.
\end{equation}
Since $x^{k+1}\in H(\ox^k,{\ov}^k)$, we get
$0\geq \langle T(\ox^{k})+{\ov}^{k},x^{k+1}-\ox^{k}\ra$,
where ${\ov}^k$ and $\ox^k$ are obtained in Steps~2 and~3, respectively. By the formulas of $\ox^k$ in Step~3 and \eqref{zk212*-c}, we derive
\begin{equation}\label{e:optimalidad2a}
\begin{aligned}
0&\geq \langle T(\ox^{k})+{\ov}^{k},x^{k+1}-x^{k}\rangle + \alpha_{k}\big{\la} T(\ox^{k})+{\ov}^{k},x^{k}-z^{k}\big{\ra}\\
&\geq \langle T(\ox^{k})+{\ov}^{k},x^{k+1}-x^{k}\ra+\alpha_{k}\delta\la
T (x^k)+\alpha_k u^k, x^k-z^k \ra.
\end{aligned}
\end{equation}
Next, Fact~\ref{proj}\ref{proj-iii} implies $\|x^k-z^k\|^2\leq \beta_k\langle T (x^k)+\alpha_k u^k, x^k-z^k\ra$. Thus, combining with \eqref{e:optimalidad2a} yields
\begin{equation}\label{eq22}
\begin{aligned}
\frac{\alpha_{k}\delta}{\beta_k}\| x^k-z^k \|^2 
&\leq \alpha_k\delta\langle T (x^k)+\alpha_k u^k, x^k-z^k\rangle \\
&\leq -\langle T(\ox^{k})+{\ov}^{k},x^{k+1}-x^{k}\ra
\leq\|T(\ox^{k})+{\ov}^{k}\|\cdot\|x^{k+1}-x^{k}\|.
\end{aligned}
\end{equation}
Choosing a subsequence $(i_k)$ such that the subsequences
$(\alpha_{i_k})_{k\in \NN}$, $(u^{i_k})_{k\in \NN}$,
$(\beta_{i_k})_{k\in \NN}$, $(x^{i_k})_{k\in \NN}$ and
$({\ov}^{i_k})_{k\in \NN}$ converge to $\tilde{\alpha}$,
$\tilde{u}$, $\tilde{\beta}$, $\tilde{x}$, and $\tilde{v}$,
respectively (this is possible by the boundedness of these sequences). Using \eqref{dif:to:0} and taking the limit in \eqref{eq22} along $(i_k)_{k\in\NN}$,
we get
\begin{equation}\label{zero22}
\lim_{k\to \infty}\alpha_{i_k}\|x^{i_k} - z^{i_k}\|^2=0.
\end{equation}
Now we consider two cases,

\noindent {\bf Case~1:} $\disp\lim_{k\to\infty} \alpha_{i_k}=\tilde{\alpha} > 0$. By \eqref{zero22}, $
\disp\lim_{k\to \infty}\|x^{i_k} - z^{i_k}\|^2=0$. By continuity of the projection,
we have $\tilde{x}=P_C\big(\tilde{x}-\tilde{\beta}(T(\tilde{x})+\tilde{\alpha}\tilde{u})\big)$. So, $\tilde{x}\in S_*$ by Proposition~\ref{parada}.

\noindent {\bf Case~2:} $\disp\lim_{k\to\infty} \alpha_{i_k} = 0$. Similar to the proof of Theorem \ref{teo12}, we also obtain $\tilde{x}\in S_*$.

Thus, all accumulation points of $(x^k)_{k\in\NN}$ are in $S_*$. 
\endproof

\medskip
Finally, by reasoning analogously to the proof of Theorem~\ref{Teo-F}, we derive the convergence result.

\begin{theorem}\label{l:optimalidad2b} The sequence $(x^k)_{k\in \NN}$ converges to $\ox=P_{S_*}(x^0)$.
\end{theorem}

\section{An Example}
\label{s:expl}

In this section, we apply the proposed algorithms (with and without normal vectors) to an instance of problem \eqref{prob}. We will see that the use of normal vectors to the feasible set might be beneficial.
\begin{example}\label{ex:1}
Let $B:=(b_1,b_2)\in\RR^2$ recall that the (clockwise) rotation with angle $\gamma\in[-\pi/2,\pi/2]$ around $B$ is given by
\begin{equation*}
\mathcal{R}_{\gamma,B}:\RR^2\to\RR^2:x\mapsto
\bigg[\,\begin{aligned}
\cos\gamma &\ \ \ \sin\gamma\\
-\sin\gamma &\ \ \ \cos\gamma
\end{aligned}\,\bigg](x-B)+B,
\end{equation*}
We consider problem \eqref{prob} in $\RR^2$ with the operator $
T:=\mathcal{R}_{-\tfrac{\pi}{2},B}-\Id$  where $B:=(\tfrac{1}{2},1)$, and the feasible set is given as
 $$C:=\left\{(x_1,x_2)\in \RR^2 :\ \ x_1^2+x_2^2\leq 1,\ x_1\leq 0,\ x_2\geq 0\right\}.$$
\end{example}
Note that operator $T$ is Lipschitz continuous with constant $L=2$, but not monotone. Now we prove that $T$ satisfies \ref{a2}, i.e., $\Sdual=S_*$. Let
us split our analysis into two parts.

\noindent {\bf Part~1:} (The primal problem has a unique solution). For $x:=(x_1,x_2)\in\RR^2$, consider the operator
\begin{equation}\label{e:150130a}
T(x):=
\left[\begin{matrix}
0 & -1\\1 & 0
\end{matrix}\right](x-B)+B-x
=
\left[\begin{matrix}
-1 & -1\\1 & -1
\end{matrix}\right]x+
\left[\begin{matrix}
3/2\\1/2
\end{matrix}\right].
\end{equation}
We will show that the primal variational inequality problem \eqref{prob}, has a unique solution.
Indeed, notice that the solution (if exists); cannot lie in the interior of $C$ (because $T(x)\neq 0$ for all $x\in C$); and also
cannot lie on the two segment $\{0\}\times[0,1]$ and $[-1,0]\times\{0\}$ (by direct computations).
Thus, the solution must lie on the arc $\Gamma:=\{(x_1,x_2)\in\RR^2\,|\,x_1^2+x_2^2=1,x_1\leq 0,x_2\geq0\}.$ Using polar coordinates, set $x=(\cos t,\sin t)\in\Gamma$, $t\in(\pi/2,\pi)$. Then,
$$
T(x)=\left[\begin{matrix}
-\cos t-\sin t+\tfrac{3}{2}\\\cos t-\sin t+\tfrac{1}{2}
\end{matrix}\right].
$$
Since $x_*\in S_*$ , the vectors $x_*$ and $T(x_*)$ must be parallel. Hence,
\begin{subequations}
\begin{align}
\frac{-\cos t_*-\sin t_*+\tfrac{3}{2}}{\cos t_*}&=\frac{\cos t_*-\sin t_*+\tfrac{1}{2}}{\sin t_*}\nonumber\\
-\sin t_*\cos t_*-\sin^2 t_*+\tfrac{3}{2}\sin t_*&=
\cos^2 t_*-\cos t_*\sin t_*+\tfrac{1}{2}\cos t_*\nonumber\\
\tfrac{3}{2}\sin t_*-\tfrac{1}{2}\cos t_*&=1\nonumber\\
\tfrac{3}{\sqrt{10}}\sin t_*-
\tfrac{1}{\sqrt{10}}\cos t_* &=\tfrac{2}{\sqrt{10}}\nonumber\\
\sin\left(t_*-\arcsin(\tfrac{1}{\sqrt{10}})\right)&=\tfrac{2}{\sqrt{10}}.\nonumber
\end{align}
\end{subequations}
Since $t\in(\pi/2,\pi)$ for all $x\in C$, we have
$
t_*=\pi-\arcsin(\tfrac{2}{\sqrt{10}})+\arcsin(\tfrac{1}{\sqrt{10}})\approx 2.7786.
$
Then, the unique solution is
$
x_*=(\cos t_*,\sin t_*)\approx(-0.935,0.355).
$

\noindent {\bf Part~2:} (The primal solution is also a solution of
the  dual problem). Now, we will show that $x_*$ is a solution to
the dual problem and as consequence of the continuity of $T$ and
Fact~\ref{by-cont} the result follows.  If $x_*\in \Sdual$, $
\langle T(y),y-x_*\rangle \geq 0$ for all $y\in C$. First, notice that $\|x_*\|=1$ and
\begin{equation*}
T(x_*)\approx\left[\begin{matrix}
-1 & -1\\1 & -1
\end{matrix}\right]
\left[\begin{matrix}
-0.935\\
0.355
\end{matrix}\right]
+
\left[\begin{matrix}
3/2\\1/2
\end{matrix}\right]
\approx
\left[\begin{matrix}
2.08\\-0.79
\end{matrix}\right]
\approx -2.22\,x_*.
\end{equation*}
So, we can write
\begin{equation}\label{e:150130c}
T(x_*)=\gamma(-x_*)\quad
\text{where}\quad 2<\gamma\approx 2.22.
\end{equation}
On the other hand, from \eqref{e:150130a}, we can check that $
\langle T(y)-T(x_*),y-x_*\rangle =-\|y-x_*\|^2, \quad  \forall y\in\RR^2$. 
(This is why $T$ is never monotone!). It follows that $
\langle T(y),y-x_*\rangle =\langle T(x_*),y-x_*\ra-\|y-x_*\|^2$. 
Thus, it suffices to prove
\begin{equation}\label{e:150130b}
\langle T(x_*),y-x_*\rangle \geq\|y-x_*\|^2
\quad \text{for all}\quad y\in C.
\end{equation}
Take $y\in C$, so $\|y\|\leq 1$. we define $z=\disp \frac{x_*+y}{2}$. Then,
\[
\langle z,z-x_*\rangle =
\tfrac{1}{2}\langle y+x_*,z-x_*\rangle =
\tfrac{1}{4}\langle y+x_*,y-x_*\rangle =\tfrac{1}{4}(\|y\|^2-\|x_*\|^2)\leq 0,
\]
implying that $
\langle z-x_*,z-x_*\rangle =\langle z,z-x_*\rangle +
\langle -x_*,z-x_*\rangle \leq \langle -x_*,z-x_*\ra$.
Combining the last inequality with the definition of $z$, we get
\begin{align*}
0&\leq\|y-x_*\|^2=4\|z-x_*\|^2=4\langle z-x^*,z-x_*\rangle \leq 4\langle -x_*,z-x_*\rangle \\
&=2\langle -x^*,y-x_*\ra
<\gamma\langle -x_*,y-x_*\rangle =\langle \gamma(-x_*),y-x_*\ra
=\langle T(x_*),y-x_*\rangle ,
\end{align*}
where we use \eqref{e:150130c} in the last inequality. This proves
\eqref{e:150130b} and thus complete the proof. Consequently, $T$ satisfies \ref{a2} and the unique solution of the problem is
$x_*\approx(-0.935,0.355)$.

We now apply the proposed algorithms (with and without normal vectors) to the above problem. In Figures~\ref{fig:B1}--\ref{fig:F3} below, we show the first {\em five} iterations of sequences $(y^k)_{k\in\NN}$ (generated without normal vectors) and $(x^k)_{k\in\NN}$ (generated with nonzero normal vectors). 
\begin{figure}[h!]
\centering
\begin{minipage}{.49\textwidth}
\includegraphics[width=.9\textwidth]{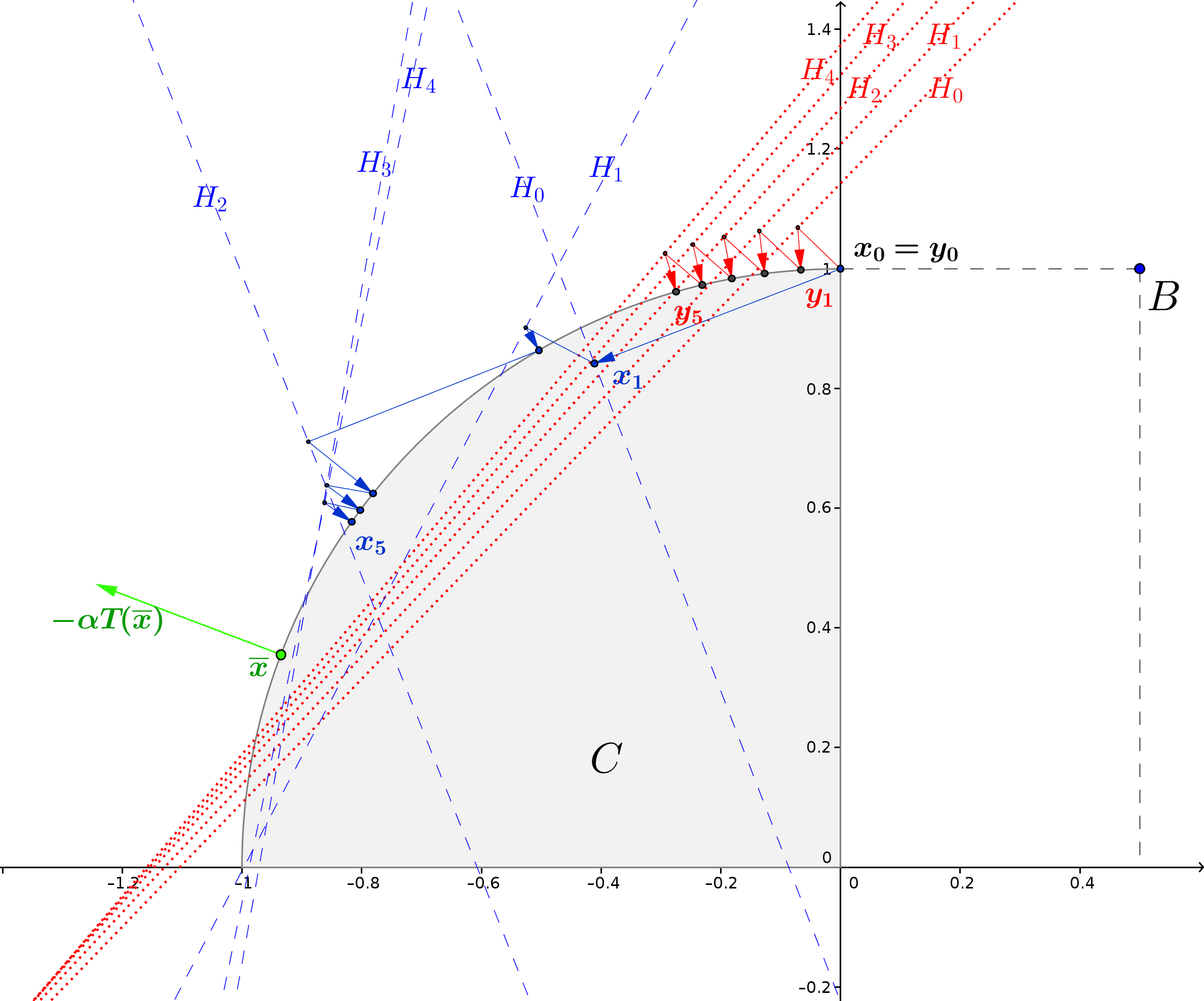}
\caption{{\bf Variant~B.1}.}
\label{fig:B1}
\end{minipage}
\begin{minipage}{.49\textwidth}
\includegraphics[width=.9\textwidth]{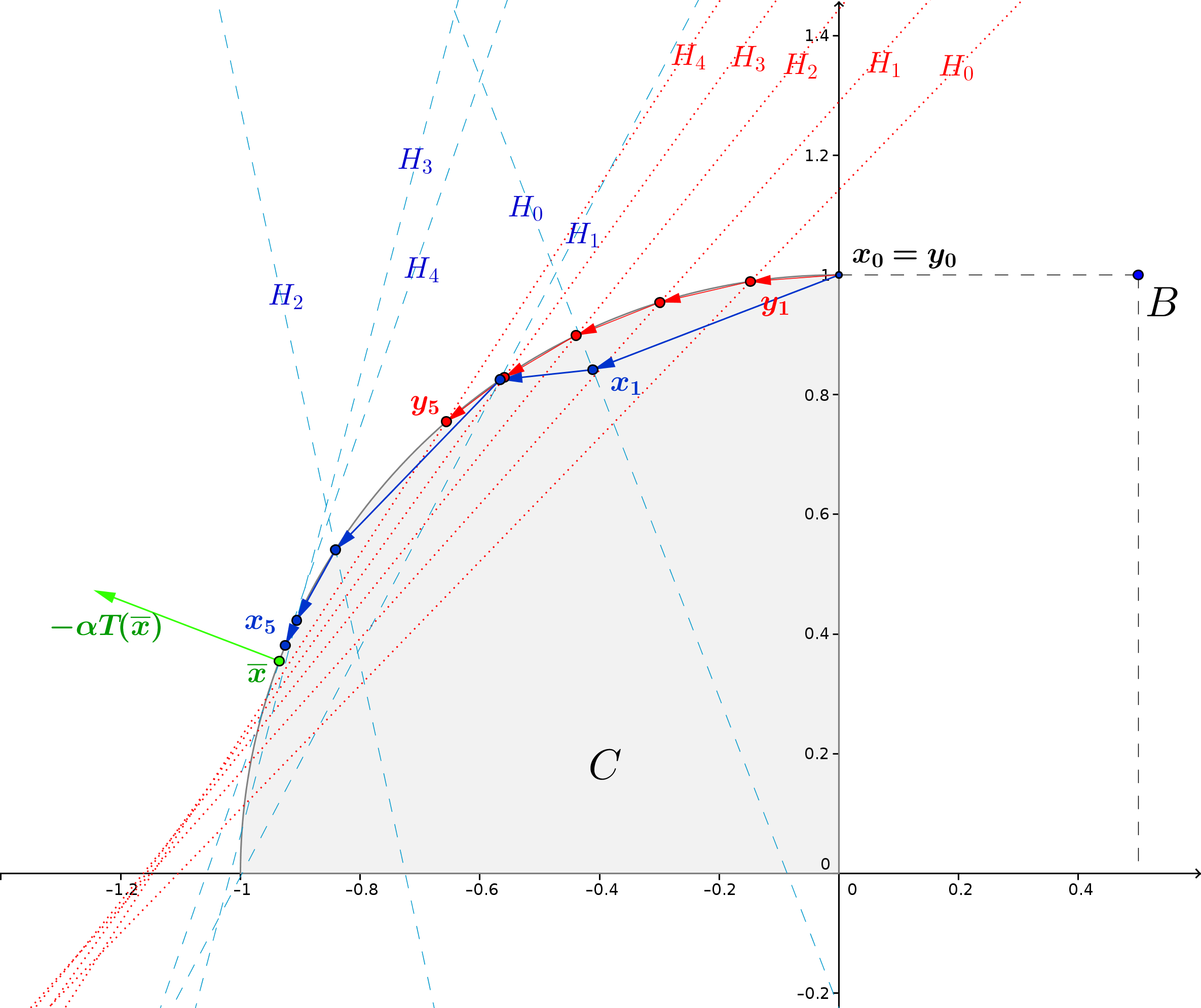}
\caption{{\bf Variant~B.2}.}
\end{minipage}
\end{figure}
\begin{figure}[h!]
\centering
\begin{minipage}{.49\textwidth}
\includegraphics[width=.9\textwidth]{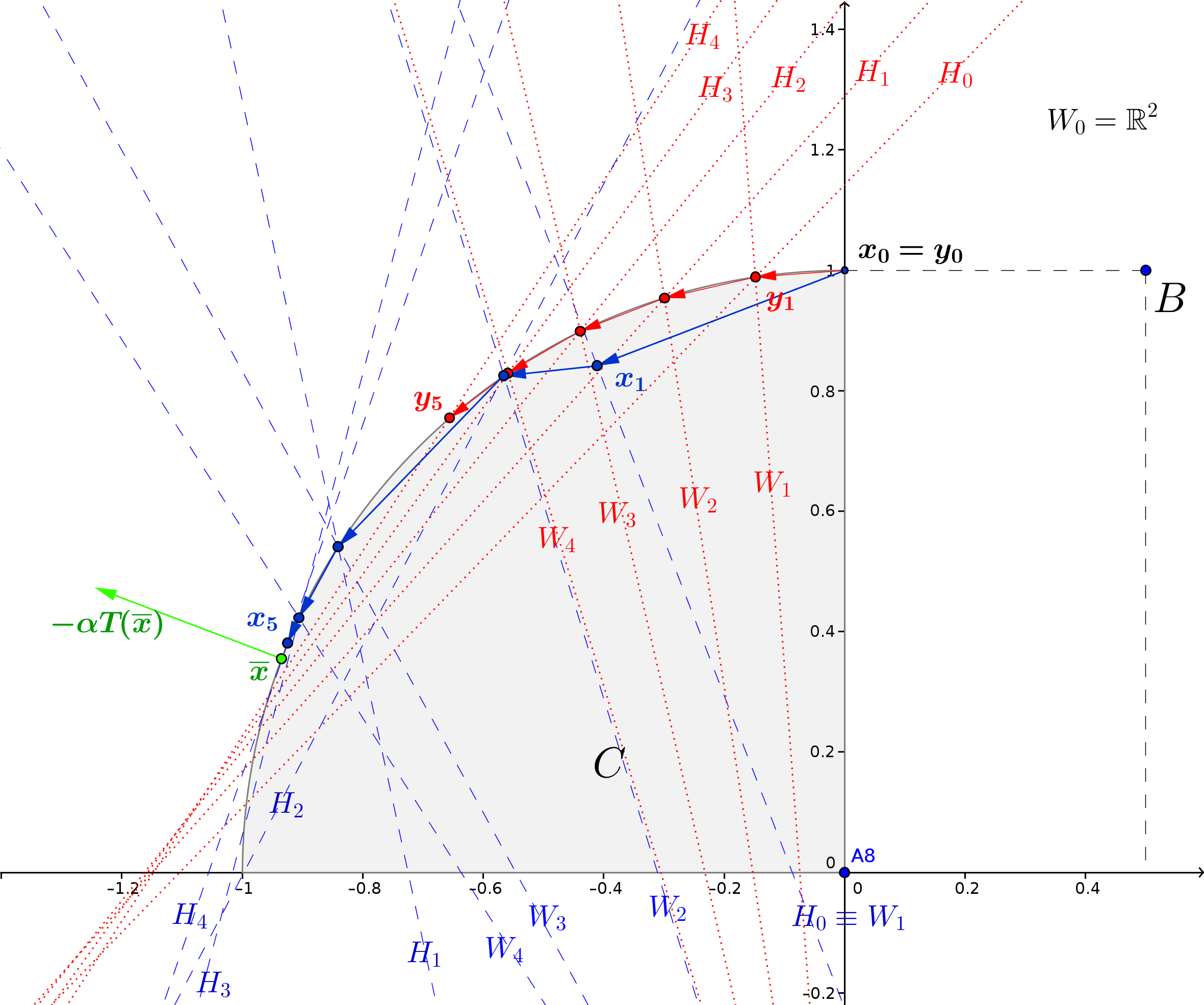}
\caption{{\bf Variant~B.3}.}
\end{minipage}
\begin{minipage}{.49\textwidth}
\includegraphics[width=.9\textwidth]{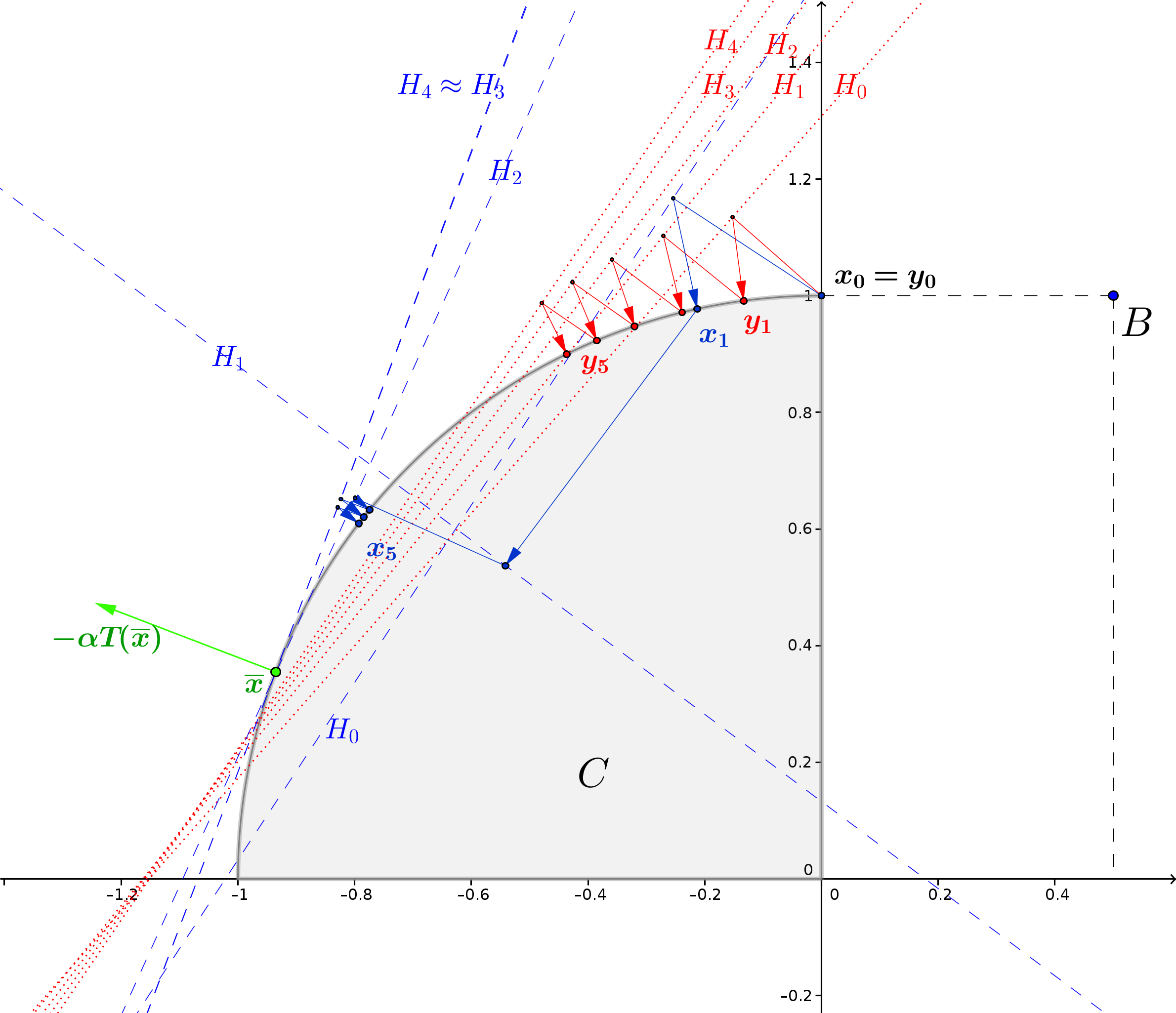}
\caption{{\bf Variant~F.1}.}
\end{minipage}
\end{figure}
\begin{figure}[h!]
\centering
\begin{minipage}{.49\textwidth}
\includegraphics[width=.9\textwidth]{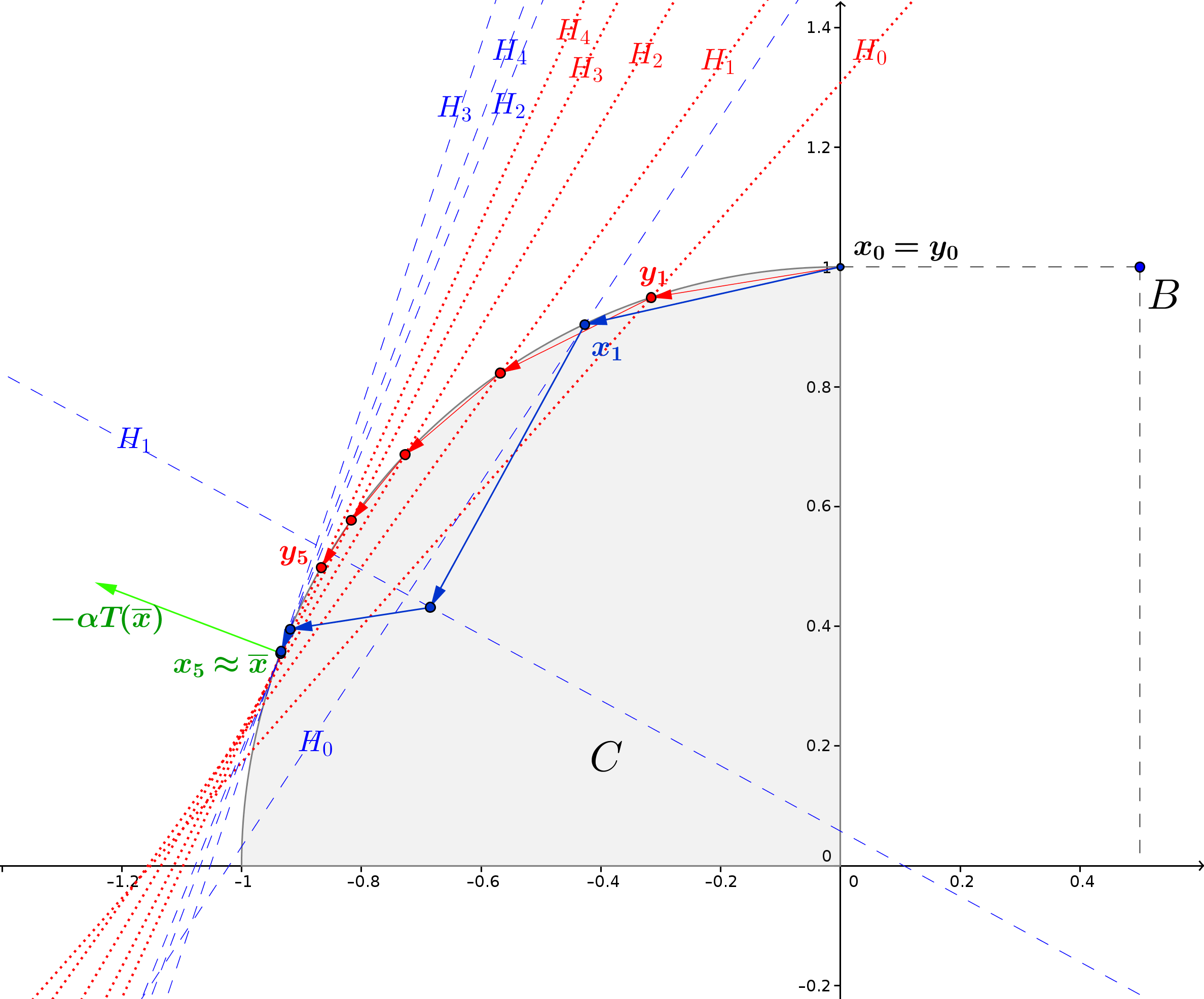}
\caption{{\bf Variant~F.2}.}
\end{minipage}
\begin{minipage}{.49\textwidth}
\includegraphics[width=.9\textwidth]{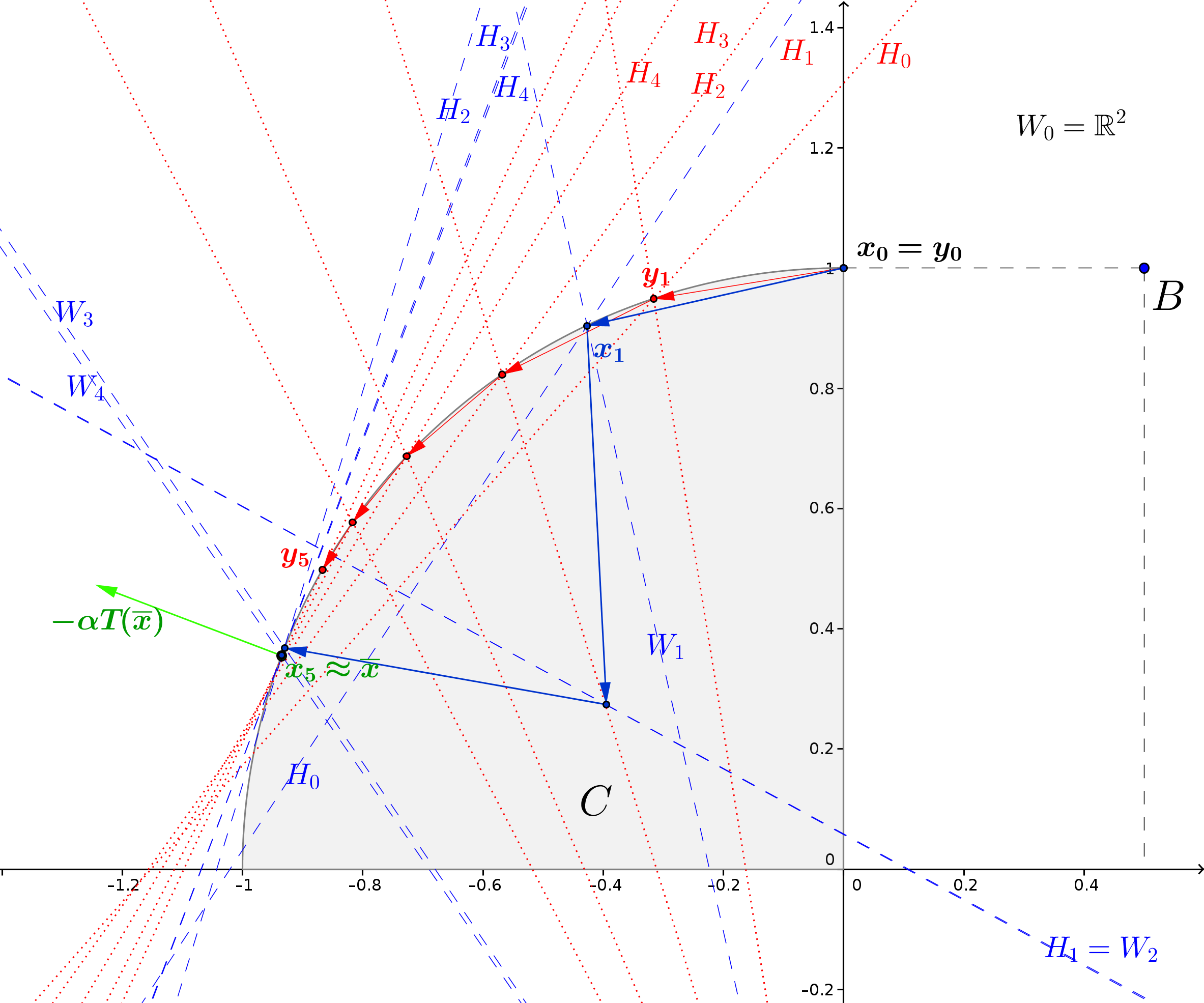}
\caption{{\bf Variant~F.3}.}
\label{fig:F3}
\end{minipage}
\end{figure}
The performance suggests that our approach can be used in a {\em hybrid} scheme that takes advantage of normal vectors in early iterations.
\section{Conclusion}\label{sec-7}
In this paper, we have proposed two conceptual conditional extragradient algorithms that generalize classical extragradient algorithms for solving constrained variational inequality problems (VIP). The main idea is to use nonzero normal vectors to the
feasible set to improve the convergence. This approach uses two different linesearches extending several known projection algorithms for VIP. 
These linesearches allow us to find suitable halfspaces containing the solution set of the problem by using nonzero normal vectors of the feasible set. It is well-known in the literature that such procedures are very effective in absence of Lipschitz continuity exploiting most of the information available at each iteration to produce possibly long steplengths. Convergence results are also established assuming existence of solutions, continuity and a weaker condition than pseudomonotonicity on the operator enlarging the class of VIP that we can solve. This is a humble attempt in targeting more efficient variants which may permit to find the optimal choice of normals on the feasible set. 

Several of the ideas of this paper merit further investigation, some of which would be presented in future work. In particular, we are working on variants of the projection algorithms proposed in \cite{yunier-reinier-S} for solving nonsmooth variational inequalities. The difficulties of extending this previous result to point-to-set operators are non-trivial, the main obstacle lies in the impossibility to use linesearches or separating techniques. To the best of our knowledge, variants of the linesearches for variational inequalities require smoothness of $T$: even for nonsmooth convex optimization problems ($T=\partial f$), it is not possible make linesearch because the negative subgradients are not always descent directions. Actually, a few explicit methods have been proposed in the literature for solving
nonsmooth monotone variational inequality problems (see, e.g., \cite{regina-svaiter-2005,he}). Moreover, future work will address further investigation on the modified Forward-Backward splitting iteration for inclusion
problems \cite{yu-re-2, yun-reinier-1, Tseng}, exploiting the
additive structure of the main operator and adding dynamic choices of the stepsizes with conditional and deflected techniques \cite{cond,francolli}. 

\section*{Acknowledgments}
JYBC was partially supported by a startup research grant of Northern Illinois University and by the National Science Foundation grant DMS-1816449. HMP was partially supported by Autodesk, Inc. via a gift made to the Department of Mathematical Sciences, University of Massachusetts Lowell.
This work was initiated while JYBC and HMP were visiting the University of British Columbia Okanagan (UBCO). They are very grateful to the Irving K. Barber School of Arts and Sciences
at UBCO and particularly to Heinz H. Bauschke and Shawn Wang for the generous hospitality. The authors also thank the anonymous  referees  for  their  valuable  suggestions.

\end{document}